\documentclass[11pt,reqno]{amsart}

\theoremstyle{plain}
\newtheorem{theorem}{Theorem}[section]
\newtheorem{lemma}[theorem]{Lemma}

\newtheorem{proposition}[theorem]{Proposition}

\newtheorem{definition}[theorem]{Definition}
\newtheorem{example}[theorem]{Example}

\newtheorem{remark}[theorem]{Remark}
\theoremstyle{remark}

\newcommand{\C}{C^{\infty}}
\newcommand{\R}{\mathbb{R}}
\newcommand{\Z}{\mathbb{Z}}
\newcommand{\sh}{\mathcal{O}}
\newcommand{\ev}{_{\overline{0}}}
\newcommand{\od}{_{\overline{1}}}

\usepackage{amssymb,amsmath,amsthm,mathrsfs}
\usepackage{epsfig,amssymb,amsmath,amsthm,graphicx,psfrag}
\usepackage{MnSymbol}
\usepackage{tikz-cd} 
\usepackage{enumerate}
\usepackage{comment}
\usepackage{enumitem}
\usepackage{array}
\usepackage{multirow}
\usepackage[hyperindex,bookmarksnumbered,plainpages]{hyperref}

\title{The structure of $\C$-superschemes}

\author{Cristian Danilo Olarte}

\address{Instituto de Matemáticas, FCEyN, Universidad de Antioquia, 50010  Medellín, Colombia}

\email{cristian.olarte@udea.edu.co}

\author{Pedro Rizzo} 

\email{pedro.hernandez@udea.edu.co}

\author{Alexander Torres-Gomez}
 
\email{galexander.torres@udea.edu.co}

\begin{document}

\begin{abstract}
This paper establishes a structural generalization of Batchelor’s theorem within the framework of $C^\infty$-superschemes. Our main result proves that any Batchelor space satisfies a global splitness condition, establishing an isomorphism between the structure sheaf and its associated graded sheaf. Although this isomorphism is non-canonical, the existence of a splitting endows the structure sheaf with a natural $\mathbb{Z}_{\geq 0}$-grading. This grading is shown to be equivalent to the data of an even superderivation, which we term an Euler vector field. Consequently, global splittings of $C^\infty$-superspaces can be characterized in terms of Euler vector fields, providing a differential-geometric formulation of the splitting.\\

\end{abstract}


\maketitle

\noindent

\tableofcontents

\section{Introduction}

In a seminal paper \cite{Batchelor}, Batchelor proved that every smooth supermanifold is globally split. This means its structure sheaf is of the form $\bigwedge_{C^\infty_M} \mathcal{E}$, where $\mathcal{E}$ is a locally free sheaf over the sheaf of smooth functions on a smooth manifold $M$. Consequently, every smooth supermanifold is non-canonically isomorphic to a ``twisted" (or non-commutative) version of a standard vector bundle over a smooth manifold. In this paper, our primary goal is to generalize Batchelor’s theorem to $C^\infty$-superspaces, which extend the class of smooth supermanifolds.\\

A $C^\infty$-superspace \cite[Definition 4.26]{ORTG} is a ringed space consisting of a topological space equipped with a sheaf of $C^{\infty}$-superrings. These are the geometric spaces associated with $C^{\infty}$-superrings (see \cite[Definition 3.1]{ORTG} and \cite[Definition 1.4.1]{Yau1}). Broadly speaking, $C^{\infty}$-superspaces enlarge the class of supermanifolds by allowing the structure sheaf to contain purely even nilpotent sections arising from the underlying $C^{\infty}$-ringed structure.\\

The notion of a $C^{\infty}$-superring emerges as a synthesis of superrings and $C^\infty$-rings. Formally, it is a superring whose even part is endowed with the structure of a $C^\infty$-ring \cite[Definition 3.1]{ORTG}. On one hand, superrings—and their associated geometric spaces, such as supermanifolds and superschemes—are non-commutative structures that are relatively straightforward to handle, providing a general perspective on non-commutative geometric spaces. For a comprehensive discussion and relevant references, see \cite{CCF, BRUZZO, Manin, Westra, Noja}. On the other hand, $C^\infty$-rings are $\mathbb{R}$-algebras equipped with $n$-ary operations parameterized by smooth real-valued functions on $\mathbb{R}^n$. The $C^\infty$-ring framework allows for the simultaneous study of smooth and singular spaces and enables the description of nilpotent objects often absent in traditional differential geometry. Furthermore, the category of $C^\infty$-schemes possesses superior categorical properties compared to the category of smooth manifolds; for instance, all fiber products exist. For further reading on $C^\infty$-rings and $C^\infty$-spaces, see \cite{J, MR, D2, Ler1, Olarte}.\\

Other approaches combining $C^{\infty}$-rings and superrings have been proposed by Carchedi and Roytenberg \cite{Carchedi}, as well as Nishimura \cite{NISHIMURA} and Yetter \cite{YETTER}. While Carchedi and Roytenberg provided a powerful categorical approach via ``super Fermat theories," its complexity can be daunting for geometrically oriented researchers. Conversely, Nishimura and Yetter focused primarily on constructing a theory of synthetic differential supergeometry.\\

In \cite{ORTG}, we developed an algebro-geometric framework for $C^{\infty}$-superrings and $C^{\infty}$-superschemes. A key result of that work was the equivalence between the category of fair $C^{\infty}$-superrings and the category of fair affine $C^{\infty}$-superschemes. While this mirrors results in classical commutative algebraic geometry, a current challenge in studying $C^{\infty}$-superschemes is that many algebraic concepts and constructions for $C^\infty$-superrings remain undefined. Consequently, we must not only apply algebraic constructions to understand geometric spaces but also develop the underlying algebra of these objects.\\

The central result of this paper is a structural generalization of Batchelor’s theorem for $C^\infty$-superspaces. In analogy with the theory of supermanifolds, the isomorphism between the structure sheaf of a $C^\infty$-superspace and the exterior algebra of its fermionic sheaf—referred to as the global splitting condition—is inherently non-canonical. The collection of such splittings is equivalently characterized by the existence of an even superderivation that is adapted to the filtration induced by the sheaf of ideals generated by all odd nilpotent sections (see \cite{Koszul}); we refer to this as an Euler vector field. We define and investigate this class of derivations, showing that they provide a natural characterization of the $\mathbb{Z}_{\geq 0}$-grading induced by the splitting of the structure sheaf. Consequently, a split $C^\infty$-superspace may be formally understood as a triple $(X, \mathcal{O}_{\mathscr{X}}, \mathcal{E})$, consisting of a topological space $X$, a structure sheaf of $C^\infty$-superrings $\mathcal O_{\mathscr X}$, and a specified Euler vector field $\mathcal E$.\\

We introduce the notion of a Batchelor space, defined as a $C^\infty$-superspace whose structure sheaf is both locally split and fine. These represent the minimal conditions required to prove the $C^\infty$-version of Batchelor’s theorem. The category of smooth supermanifolds is a subcategory of the category of Batchelor spaces, which in turn is a subcategory of the category of projected $C^\infty$-superspaces. Furthermore, we provide a systematic method for constructing $C^\infty$-superspaces to illustrate the boundaries of these classifications. Specifically, we construct: a $C^\infty$-superspace that is a supermanifold; a $C^\infty$-superspace that is a Batchelor space but not a supermanifold; a $C^\infty$-superspace that fails to be a Batchelor space; a $C^\infty$-superspace equipped with a fine sheaf that is not locally split; a non-projected fair $\C$-superspaces.\\

This paper is organized as follows. Section \ref{sec:two} provides a concise review of $C^\infty$-rings and $C^\infty$-superrings, addressing the compatibility of these structures with both superring theory and the $C^\infty$-ring framework. Section \ref{sec:three} presents our main findings in $C^\infty$-supergeometry, including the generalized Batchelor Theorem and the characterization of splittings via Euler vector fields.

\section{$C^\infty$-superrings}\label{sec:two}

In this section we introduce some fundamental properties of \( \C \)-superrings. For a more general context of the properties and theoretical implications of \( \C \)-superrings, we refer the reader to \cite{ORTG} where these structures are systematically analyzed. In a nutshell, a \( \C \)-superring is a novel and cohesive algebraic object that synthesizes the established frameworks of superrings and \( \C \)-rings. 

For the sake of completeness, before introducing the constructions related to $C^\infty$-superrings, we present a brief overview of superrrings and $C^\infty$-ring theory. 

For us, a superring is a $\Z_2$-graded unitary supercommutative ring. For a more comprehensive treatment and recent advances in the theory of superrings and its relation with supermanifolds see, e.g., \cite{Manin} and \cite{Westra}. Now, a \emph{\( C^\infty \)-ring} is a generalization of the algebra of smooth functions on a manifold. Using $\C$-rings instead of rings leads to $\C$-algebraic geometry. This foundational structure serves as the cornerstone for synthetic and derived differential geometry, allowing the study of smooth spaces in purely categorical or algebraic settings. For a comprehensive treatment of these concepts, we refer the reader to the seminal works of Dubuc \cite{D1, D2}, Moerdijk and Reyes \cite{MR, MRI}, and, more recently, Joyce \cite{J}.

Formally, a \emph{$C^\infty$-ring} is defined as a pair $(\mathfrak{C}, \{\phi_f\})$, consisting of a set $\mathfrak{C}$ and a collection of \emph{$n$-ary operations} $\{\phi_f: \mathfrak{C}^n \to \mathfrak{C} \mid f \in C^\infty(\mathbb{R}^n), n \in \mathbb{N}\}$, which preserve both projections and compositions. Specifically, if $p_i: \mathbb{R}^n \to \mathbb{R}$ is the $i$-th coordinate projection, then $\phi_{p_i}(c_1, \ldots, c_n) = c_i$ for all $(c_1, \ldots, c_n) \in \mathfrak{C}^n$. Furthermore, for any smooth functions $h_1, \ldots, h_n \in C^\infty(\mathbb{R}^m)$ and $g \in C^\infty(\mathbb{R}^n)$, the map associated with their composition satisfies $\phi_{g(h_1, \ldots, h_n)}(c_1, \ldots, c_m) = \phi_g(\phi_{h_1}(c_1, \ldots, c_m), \ldots, \phi_{h_n}(c_1, \ldots, c_m))$, for all $(c_1, \ldots, c_m) \in \mathfrak{C}^m$. 
 
A \textit{morphism} of $C^\infty$-rings from $(\mathfrak{C}, \{\phi_f\})$ to $(\mathfrak{D}, \{\psi_f\})$ is a set-theoretic map $\varphi: \mathfrak{C} \to \mathfrak{D}$ that commutes with the smooth operations, meaning $\varphi(\phi_f(c_1, \ldots, c_n)) = \psi_f(\varphi(c_1), \ldots, \varphi(c_n))$ for every $n \in \mathbb{N}$ and $f \in C^\infty(\mathbb{R}^n)$. Together, $C^\infty$-rings and their corresponding morphisms form a category denoted by $\mathbf{C^\infty Rings}$. 
 
Within this category, a $C^\infty$-ring $\mathfrak{C}$ is said to be \textit{finitely generated} if there exists a finite subset of generators $\{c_1, \ldots, c_n\} \subseteq \mathfrak{C}$ such that every element $c \in \mathfrak{C}$ can be expressed as $c = \phi_f(c_1, \ldots, c_n)$ for some $f \in C^\infty(\mathbb{R}^n)$. 
 
The archetypal example of a $C^\infty$-ring is the algebra $C^\infty(M)$ of smooth real-valued functions on a differentiable manifold $M$. This geometric instance is highly representative of the broader algebraic theory, as every finitely generated $C^\infty$-ring is isomorphic to a quotient of $C^\infty(\mathbb{R}^n)$ by an appropriate ideal. Beyond this classical setting, the theory contains a diverse range of examples originating in functional analysis, algebraic geometry, and differential topology, illustrating its structural flexibility. Among these structures, a $C^\infty$-ring $\mathfrak{R}$ is classified as \textit{fair} if it is both finitely generated and germ-determined. The latter condition requires the family of localization morphisms $\{\mathcal{L}_x: \mathfrak{R} \to \mathfrak{R}_x\}_{x \in X_{\mathfrak{C}}}$ to be jointly injective, meaning that $\bigcap\limits_{x \in X_{\mathfrak{C}}} \ker(\mathcal{L}_x) = \{0\}$.
 
Finally, we extend this framework to a geometric setting by defining a $C^\infty$-ringed space as a pair $(X, \mathcal{O}_X)$, where $X$ is a topological space and $\mathcal{O}_X$ is a sheaf of $C^\infty$-rings on $X$. For any point $p \in X$, the corresponding stalk $\mathcal{O}_{X,p}$, which is given by the direct limit of the $C^\infty$-rings of sections over open neighborhoods containing $p$, is also a $C^\infty$-ring.

We are now ready to introduce one of the main concepts of this article. 
\begin{definition}\label{def:superringnew}
A $\C$-superring is a superring $\mathfrak{R}=\mathfrak{R}\ev\oplus \mathfrak{R}\od$ whose even part $\mathfrak{R}\ev$ is endowed with a $\C$-ring structure that is compatible with the underlying ring structure of $\mathfrak{R}\ev$ (induced by $\mathfrak{R}$).
\end{definition}	

\begin{definition}
Let $\mathfrak R=\mathfrak R\ev \oplus \mathfrak R\od$ be a $C^\infty$-superring. The superideal $\mathfrak J_{\mathfrak R}$ generated by the odd part $\mathfrak R\od$, that is, $\mathfrak J_{\mathfrak R}=\mathfrak R\, \mathfrak R\od=\mathfrak R\od^2 \oplus \mathfrak R\od$, is called the canonical superideal of $\mathfrak R$. The quotient ring $\overline{\mathfrak R}= \mathfrak R/\mathfrak J_{\mathfrak R}$ is called the superreduced $\C$-ring of $\mathfrak R$.
\end{definition}

The superreduced $\C$-ring $\overline{\mathfrak R}= \mathfrak R/\mathfrak J_{\mathfrak R}$ is a commutative ring isomorphic to $\mathfrak R\ev/\mathfrak R\od^2$. Therefore, $\overline{\mathfrak R}$ inherits the $\C$-ring structure since it is a quotient of a $\C$-ring by an ideal. Furthermore, if $\mathfrak J_{\mathfrak R}$ is finitely generated, then $\mathfrak J_{\mathfrak R}/\mathfrak J_{\mathfrak R}^2$ is a finitely generated $\overline{\mathfrak  R}$-module.

\begin{definition}\label{def:smorph}
Let $\mathfrak{R}=\mathfrak{R}\ev\oplus \mathfrak{R}\od$ and $\mathfrak S=\mathfrak S\ev\oplus \mathfrak S\od$ be $\C$-superrings. A morphism of $\C$-superrings $\varphi:\mathfrak{R}\rightarrow \mathfrak S$ is a grade-preserving superring morphism such that its restriction to the even part $\varphi|_{\mathfrak{R}\ev}:\mathfrak{R}\ev\rightarrow \mathfrak S\ev$, is also a morphism of $\C$-rings.
\end{definition}

\begin{example}\label{ex:first}
As a first illustrative example of a $\C$-superring we consider $\C(\R^{1|2})$ the superring of smooth functions on the Euclidian superspace $\R^{1|2}$, that is, the superring of smooth functions extended by two anticommuting variables $\theta^1, \theta^2$. Its even part is the $\R$-algebra $\C(\R^{1|2})\ev$ whose elements are super polynomials of the form $F=f +g \theta^1\theta^2$, where $f,g\in\C(\R)$.

Given $F_i=f_i+g_i\theta^1\theta^2\in \C(\R^{1|2})\ev$ for $i=1,\ldots,n$ and a smooth function $h\in\C(\R^n)$, we define the associated $n$-ary operation $\phi_h$ by
\begin{equation}\label{eq:firstex}
\phi_h(F_1,\ldots,F_n)=h(f_1,\ldots,f_n)+\displaystyle{\sum_{j=1}^n\partial_jh(f_1,\ldots,f_n)g_j\theta^1\theta^2}. 
\end{equation}
By the chain rule, and because the derivatives of projection maps are constant, this formula defines a $\C$-ring structure on $\C(\R^{1|2})\ev$, that is compatible with its underlying multiplication. Indeed, we only verify the compatibility of the product, as the case for addition is straightforward. Then, taking $h(x,y)=xy$, to represent standard multiplication, we have $\partial_x h=y$, $\partial_y h=x$. Therefore, for $F_1=f_1 +g_1 \theta^1\theta^2$ and $F_2=f_2 +g_2 \theta^1\theta^2$, equation \eqref{eq:firstex} and the relation $(\theta^1\theta^2)^2 = 0$ yield
$$
\phi_h(F_1, F_2)=f_1f_2+f_2g_1\theta^1\theta^2+f_1g_2\theta^1\theta^2=(f_1+g_1\theta^1\theta^2)\cdot(f_2+g_2\theta^1\theta^2)=F_1\cdot F_2
$$
which proves the compatibility of the structures.
\end{example}

In fact, Example \ref{ex:first} admits a natural generalization. Given a $\C$-ring $\mathfrak{C}$ and a finitely generated $\mathfrak C$-module $\xi$ with basis $\{\theta^1,\ldots, \theta^q\}$, we can construct the polynomial superring over $\mathfrak{C}$. This superring is defined by the exterior algebra $\mathfrak{C}[\theta^1, \ldots, \theta^q]^{\pm}:=\bigwedge_{\mathfrak C} \xi$, where $\theta^1,\ldots, \theta^q$ are the odd variables over $\mathfrak C$. 

The proposition below establishes that $\mathfrak{C}[\theta^1, \ldots, \theta^q]^{\pm}$ inherits a natural $\C$-superring structure. To this end, we highlight the following key fact: every element $F\in\mathfrak{R}\ev$ can be uniquely written as $F=f+N$, where $f\in\mathfrak{C}$, $N=\sum\limits_{|I|:\,\text{even}} f^I\theta^I$, for $I\subseteq\{1,\ldots,q\}$ non-empty and $|I|$ denoting the cardinality of $I$. Note that $N\in \mathfrak{R}\od^2$. Let $f:=\overline{F}$ \emph{the superreduced of} $F$ and $N:=\mathrm{nil}(F)$ \emph{the nilpotent part of $F$ coming from the odd generators}.

\begin{proposition}\label{prop:superstruct}
Under the above notations, given $\mathfrak C$ a $\C$-ring and let $\xi$ be a finitely generated $\mathfrak C$-module with basis $\{\theta^1,\ldots,\theta^q\}$. The exterior algebra $\mathfrak{R}:=\mathfrak{C}[\theta^1, \ldots, \theta^q]^{\pm}=\bigwedge_{\mathfrak C} \xi$ carries a natural $\C$-superring structure defined as follows. For $h\in\C(\R^k)$ and $F_1,\ldots,F_k\in\mathfrak R\ev$, written as $F_j=f_j+N_j$, we define the $k$-ary operation
\begin{equation}\label{eq:taylorPhi1}
\Phi_h(F_1,\ldots,F_k):=\sum_{n\geq 0}\frac{1}{n!}\sum_{j_1,\ldots,j_n=1}^{k}\phi_{\partial_{j_1}\cdots\partial_{j_n}h}(f_1,\ldots,f_k)\;N_{j_1}\cdots N_{j_n}. 
\end{equation}
This construction satisfies the following properties:
\begin{enumerate}
\item[(i)] The sum \eqref{eq:taylorPhi1} is finite (terminating at $n=\lfloor q/2\rfloor$), and the collection $\{\Phi_h\}_{h\in\C(\R^k),\,k\in\mathbb N}$ equips $\mathfrak R\ev$ with a $\C$-ring structure.
\item[(ii)] This $\C$-ring structure is compatible with the underlying addition, ring and scalar multiplication on $\mathfrak R\ev$ induced by the exterior algebra $\bigwedge_{\mathfrak C}\xi$, that is, $\Phi_{x+y}=+$, $\Phi_{xy}=\,\cdot\,$ and $\Phi_{\lambda x}=\lambda\,(-)$ for every $\lambda\in\R$.
\end{enumerate}
In particular, $\mathfrak R=\mathfrak{C}[\theta^1,\ldots,\theta^q]^{\pm}$, equipped with the $k$-ary operations \eqref{eq:taylorPhi1}, is a $\C$-superring.  
\end{proposition}
\begin{proof}
Since each $N_j\in\mathfrak{R}\od^2$ is even, the factors in \eqref{eq:taylorPhi1} of $N_{j_1}\cdots N_{j_n}$ commute. Thus, the product depends only on the multiset of indices $\{j_1,\ldots,j_n\}$. Likewise, the operator $\partial_{j_1}\cdots \partial_{j_n}h$ is symmetric in its indices because mixed partial derivatives of smooth functions commute. Consequently, expression \eqref{eq:taylorPhi1} is well-defined. Moreover, the ideal $\mathfrak{R}\od^2$ is nilpotent with $\left(\mathfrak{R}\od^2\right)^{\lfloor q/2\rfloor+1}=0$, so any product $N_{j_1}\cdots N_{j_n}$ with $n>\lfloor q/2\rfloor$ vanishes. This establishes that the sum in \eqref{eq:taylorPhi1} is finite, proving the finiteness assertion in $(i)$. We next verify that \eqref{eq:taylorPhi1} defines a valid $k$-ary operation.

\begin{itemize}[leftmargin=*]
\item{\bf Projections are preserved:} Let $p_i:\R^k\to\R$ be the $i$-th coordinate projection map. Then $\partial_j p_i=\delta_{ij}$ is constant, so all higher-order derivatives of $p_i$ (of order $\geq 2$) vanish. Consequently, only the terms $n=0,1$ contribute to \eqref{eq:taylorPhi1}:
$$
\Phi_{p_i}(F_1,\ldots,F_k)=\phi_{p_i}(f_1,\ldots,f_k)+\sum_{j=1}^k \phi_{\delta_{ij}}(f_1,\ldots,f_k)\,N_j = f_i+N_i=F_i \, ,
$$
where we used the fact that $\mathfrak C$, being a $\C$-ring, preserves coordinate projections ($\phi_{p_i}(f_1, \ldots, f_k) = f_i$) and that $\phi_c(\cdot) = c$ for any constant $c \in \mathbb{R}$.

\item\textbf{Compositions are preserved:} Let $h\in\C(\R^m)$, $g_1,\ldots,g_m\in\C(\R^k)$. We must show that
$$
\Phi_{h(g_1,\ldots,g_m)}(F_1,\ldots,F_k)=\Phi_h\bigl(\Phi_{g_1}(F_1,\ldots,F_k),\ldots,\Phi_{g_m}(F_1,\ldots,F_k)\bigr).
$$
Write $G_i:=\Phi_{g_i}(F_1,\ldots,F_k)$. Applying \eqref{eq:taylorPhi1} to $g_i$, we obtain explicit expressions for its superreduced and nilpotent parts:
$$
\overline{G_i}=\phi_{g_i}(f_1,\ldots,f_k), \qquad
\mathrm{nil}(G_i)=\sum_{n\geq 1}\frac{1}{n!}\sum_{j_1,\ldots,j_n=1}^k \phi_{\partial_{j_1}\cdots\partial_{j_n}g_i}(f_1,\ldots,f_k)\,N_{j_1}\cdots N_{j_n} \, .
$$
Substituting these expressions into \eqref{eq:taylorPhi1} for $\Phi_h(G_1,\ldots,G_m)$ and collecting terms according to their total order $n = n_1 + \dots + n_r$ in the $N_j$'s, the coefficient of $N_{j_1}\cdots N_{j_n}$ (summed over the index distributions among the $r$ factors $\mathrm{nil}(G_{i_1}),\ldots,\mathrm{nil}(G_{i_r})$ with appropriate factorials) equals
$$
\sum_{r\geq 0}\frac{1}{r!}\sum_{i_1,\ldots,i_r=1}^m \phi_{\partial_{i_1}\cdots\partial_{i_r}h}\bigl(\phi_{g_1}(f_1,\ldots,f_k),\ldots,\phi_{g_m}(f_1,\ldots,f_k)\bigr)\cdot\Bigl[\text{coefficient of } N_{j_1}\cdots N_{j_n} \text{ in } \textstyle\prod_{l=1}^r \mathrm{nil}(G_{i_l})\Bigr].
$$
By the higher-order chain rule (Faà di Bruno's formula), this coincides, order by order in $n$, with $\phi_{\partial_{j_1}\cdots\partial_{j_n}[h(g_1,\ldots,g_m)]}(f_1,\ldots,f_k)$. Indeed, $\partial_{j_1}\cdots\partial_{j_n}[h(g_1,\ldots,g_m)]$ is a universal polynomial expression in the partial derivatives $\partial^\bullet h$ evaluated at $(g_1,\ldots,g_m)$ and those of $g_i$ evaluated at $(x_1,\ldots,x_k)$. Since $\mathfrak C$ is a $\C$-ring, applying $\phi$ to such a composite smooth function agrees with evaluating the corresponding polynomial in the images under $\phi$, by definition.  Therefore
$$
\Phi_h(G_1,\ldots,G_m)=\sum_{n\geq 0}\frac{1}{n!}\sum_{j_1,\ldots,j_n=1}^k \phi_{\partial_{j_1}\cdots\partial_{j_n}[h(g_1,\ldots,g_m)]}(f_1,\ldots,f_k)\,N_{j_1}\cdots N_{j_n}=\Phi_{h(g_1,\ldots,g_m)}(F_1,\ldots,F_k) \, ,
$$
as required.
\end{itemize}
This completes the proof of item $(i)$.

To prove item $(ii)$, namely, the compatibility with the ring structure, we first consider addition, given by $h(x,y)=x+y$. All partial derivatives of order $\geq 2$ vanish, so \eqref{eq:taylorPhi1} truncates at $n=1$:
$$
\Phi_{x+y}(F_1,F_2)=f_1+f_2+N_1+N_2=F_1+F_2 \, .
$$
Next consider scalar multiplication by $\lambda\in\R$, given by $h(x)=\lambda x$. Likewise, only the terms for $n=0,1$ contribute:
$$
\Phi_{\lambda x}(F)=\lambda f+\lambda N=\lambda F \, .
$$
Finally, let $h(x,y)=xy$ represent ring multiplication. Then, $\partial_x h=y$, $\partial_y h=x$, $\partial_x\partial_y h=\partial_y\partial_x h=1$, and all partial derivatives of order $\geq 3$ vanish. Thus, \eqref{eq:taylorPhi1} has exactly three contributions:
\begin{itemize}
\item The term $n=0$ yields $\phi_{xy}(f_1,f_2)=f_1f_2$. 
\item The term $n=1$ yields $\phi_{\partial_x h}(f_1,f_2)\,N_1+\phi_{\partial_y h}(f_1,f_2)\,N_2=f_2N_1+f_1N_2$.
\item The term $n=2$ summing over the two orderings $(j_1,j_2) = (1,2)$ and $(2,1)$ (each contributing $\phi_{\partial_x\partial_y h}(f_1,f_2) = 1$), yields
$$
\frac{1}{2!}\Bigl(\phi_{\partial_1\partial_2 h}(f_1,f_2)\,N_1N_2+\phi_{\partial_2\partial_1 h}(f_1,f_2)\,N_2N_1\Bigr)=N_1N_2 \, .
$$
\end{itemize}
Combining these terms gives
$$
\Phi_{xy}(F_1,F_2)=f_1f_2+f_2N_1+f_1N_2+N_1N_2=(f_1+N_1)(f_2+N_2)=F_1F_2 \, ,
$$
which coincides with the pre-existing product on $\mathfrak R\ev$. This completes the proof of item $(ii)$.

Bringing together items $(i)$ and $(ii)$, we conclude that the collection $\Phi=\{\Phi_h\}$ defines a $\C$-ring structure on $\mathfrak R\ev$ compatible with its underlying ring structure. By Definition \ref{def:superringnew}, this completes the proof that $\mathfrak R=\mathfrak C[\theta^1,\ldots,\theta^q]^{\pm}$ is a $\C$-superring.
\end{proof}

\begin{remark}
If one considers the space of smooth functions on the Euclidean superspace of dimension $p\vert{}q$ as an application of Proposition \ref{prop:superstruct}, defined by
$$
\C(\R^{p|q})=\C({\R}^p)[\theta^1,\ldots, \theta^q]^{\pm} \;
	 :=\;   \frac{C^{\infty}({\R}^p)\langle Z^1,\ldots, Z^q \rangle}{\left( fZ^{i}-Z^{i}f\,,\;Z^{i}Z^{j}+Z^{j}Z^{i}
\mid f \in\C({\R}^p);\, i, j \in\{1,\ldots, q\}\right)}, 
$$
where $\theta^i$ represents the equivalence class of the indeterminate $Z^i$, there appears to be a striking discrepancy in the formulation of the $k$-ary operations. Indeed, formula \eqref{eq:taylorPhi1} includes all higher-order terms in the Taylor expansion of the smooth function $h$, whereas Example \ref{ex:first} retains only the constant and linear terms. This apparent contradiction is easily resolved: for $q \geq 4$, defining the operations on $\C(\R^{p|q})$ via the truncation in Example \ref{ex:first} fails to yield a $\C$-superring under Definition \ref{def:superringnew} because structural compatibility no longer holds.

To explain this claim in detail, recall that elements of this quotient decompose uniquely as
$$
F=f+\sum_{|I|=2n}f^I\theta^I+\sum_{|I|=2n+1}g^I\theta^I
$$
where $I\subseteq \{1,2,\ldots,q\}$ is a multi-index, $n\in \Z$, $f,f^I,g^I\in \C(\R^p)$ and $\theta^I=\theta^{i_1}\theta^{i_2}\ldots\theta^{i_{|I|}}$. Thus, the even part $\C({\R}^p)[\theta^1,\ldots, \theta^q]^{\pm}\ev$ consists of elements $F$ with $g^I = 0$ for all odd $|I|$. If we attempt to define the $\C$-ring structure on $\C({\R}^p)[\theta^1,\ldots, \theta^q]^{\pm}\ev$ using the truncated formula
\begin{equation}\label{eq:old-formula}
\Phi_h(F_1,\ldots,F_k)= \phi_h (f_1, \ldots, f_k)\,
	         +\, \sum_{i=1}^k   (\phi_{\partial_i h}) (f_1, \ldots, f_k)\,N_i,
	         \qquad F_j=f_j+N_j,
\end{equation}
 for any smooth function $h:{\R}^k\rightarrow {\R}$ and $k\in {\Bbb Z}_{\ge 1}$, this construction fails to define a $\C$-ring structure on $\C(\R^p)[\theta^1,\ldots,\theta^q]^{\pm}\ev$ as soon as $q\geq 4$. In fact, \eqref{eq:old-formula} is incompatible with the pre-existing multiplication on $\C(\R^p)[\theta^1,\ldots,\theta^q]^{\pm}\ev$, a compatibility that is required of any $\C$-ring structure extending the underlying ring structure (as seen by testing \eqref{eq:old-formula} on $h(x,y)=xy$).

To demonstrate this failure explicitly, take $p=1$ and $q=4$, so that $\C(\R)[\theta^1,\theta^2,\theta^3,\theta^4]$. Consider the two even elements
$$
F_1=f_1+\theta^1\theta^2, \qquad F_2=f_2+\theta^3\theta^4 \, ,
$$
for arbitrary $f_1,f_2\in\C(\R)$. Both elements have the form $F_j = f_j + N_j$, where $N_j$ is an even nilpotent element.

Since $\mathfrak{C}[\theta^1,\ldots,\theta^q]^{\pm}\ev$ already carries the multiplication inherited from the exterior algebra $\bigwedge_{\mathfrak C}\xi$, compatibility requires that
$$
\Phi_{xy}(F_1,F_2)=F_1 F_2=f_1f_2+f_1\,\theta^3\theta^4+f_2\,\theta^1\theta^2+\theta^1\theta^2\theta^3\theta^4 \, ,
$$
where the final term is non-zero because $\theta^1, \theta^2, \theta^3, \theta^4$ are pairwise distinct odd generators, yielding $N_1 N_2 = (\theta^1\theta^2)(\theta^3\theta^4) = \theta^1\theta^2\theta^3\theta^4 \neq 0$.

On the other hand, applying the linear formula \eqref{eq:old-formula} with $h(x,y) = xy$ (where $\partial_x h = y$ and $\partial_y h = x$) yields
$$
\Phi_{xy}(F_1,F_2)= f_1f_2+ f_2\cdot N_1 + f_1\cdot N_2 = f_1f_2+f_2\,\theta^1\theta^2+f_1\,\theta^3\theta^4 \, .
$$
Comparing the two results shows that formula \eqref{eq:old-formula} entirely omits the product $N_1 N_2 = \theta^1\theta^2\theta^3\theta^4$:
$$
\Phi_{xy}(F_1,F_2)\,=\; F_1F_2 \;-\; \theta^1\theta^2\theta^3\theta^4 \;\neq\; F_1F_2 \, .
$$
Hence, \eqref{eq:old-formula} fails to reproduce the ring multiplication on $\C(\R)[\theta^1,\theta^2,\theta^3,\theta^4]$, and cannot yield a compatible $\C$-ring structure.

This discrepancy remains hidden when $q = 2$ because the even nilpotent ideal is square-zero: the only non-zero even nilpotent monomial is $\theta^1\theta^2$, and $(\theta^1\theta^2)^2 = 0$. Consequently, $N_1 N_2 = 0$ holds automatically whenever $N_1, N_2$ are multiples of $\theta^1\theta^2$, making the linear formula \eqref{eq:old-formula} to coincide with the full Taylor expansion. However, for $q \geq 4$, distinct even nilpotent monomials such as $\theta^1\theta^2$ and $\theta^3\theta^4$ have a non-zero product. Thus, the quadratic (and higher-order) terms in the Taylor expansion no longer vanish identically. Formula \eqref{eq:taylorPhi1} defines the operation used throughout the remainder of this paper whenever the $\C$-ring operations $\Phi_h$ on a split $\C$-superring, in the sense of Definition \ref{def:Csplit} below, are invoked.
\end{remark}

\begin{definition}\label{def:Csplit}
Let $\mathfrak C$ be a $\C$-ring and $\xi$ a finitely generated $\mathfrak C$-module with basis $\{ \theta^1, \cdots, \theta^q \}$. A $\C$-superring $\mathfrak R$ is said to be \emph{split} if it is isomorphic to $\mathfrak{C}[\theta^1, \ldots, \theta^q]^{\pm}=\bigwedge_{\mathfrak C} \xi$.
\end{definition}

Observe that if $\mathfrak{R}$ is a $\C$-superring split, then it can be written as $\mathfrak{R}=\overline{\mathfrak{R}}\oplus \mathfrak{J}_{\mathfrak{R}}$. This is equivalent to claiming that the short exact sequence 
\begin{equation}\label{eq:split1}  
0\longrightarrow \mathfrak{J}_{\mathfrak{R}} \stackrel{\iota} \longrightarrow \mathfrak{R}\stackrel{\pi}{\longrightarrow} \overline{\mathfrak{R}}\longrightarrow 0
\end{equation} 
splits.

\begin{example}\label{nonsplit}
Let $\mathfrak{I}$ be the ideal of $\C(\R^{1|2})$ generated by $x^2+ \theta^1\theta^2$. Consider the $\C$-superring $\mathfrak{R}:=\C(\R^{1|2})/\mathfrak{I}$. We will show that $\mathfrak{R}$ is not split.

Recall that elements of $\mathfrak{R}=\C(\R^{1|2})$ are of the form 
$f_0(x)+f_{12}(x)\; \theta^1\theta^2 +f_1(x)\;\theta^1+f_2(x)\;\theta^2$, where $f_0, f_1, f_2, f_{12} \in\C({\R})$. Note that the ideal generated by the odd parts, $\mathfrak{R}\od^2$, can be identified with $(\overline{x}^2)$ because in the quotient $\theta^1 \theta^2$ is equivalent to $-x^2$. With this identification, $x^2$ becomes nilpotent in the quotient $\mathfrak{R}$.  In particular, $\mathfrak{R}\ev$ is isomorphic to $\C(\R)/(x^4)$. Likewise, the superreduced ring $\overline{\mathfrak R}\simeq\mathfrak{R}\ev/\mathfrak{R}\od^2$ is isomorphic to $\C(\R)/(x^2)$.  Consequently, we have a surjective morphism:
$$\C(\R)/(x^4)\rightarrow\frac{\C(\R)/(x^4)}{(x^2)/(x^4)})=\frac{\C(\R)}{(x^2)} \, .$$
This morphism corresponds to the canonical projection $\mathfrak{R}\ev\rightarrow \overline{\mathfrak{R}}$, which induces the natural graded-preserving map $\mathfrak{R}\rightarrow \overline{\mathfrak{R}}$. However, there is no injective morphism:
$$\C(\R)/(x^2)\rightarrow \C(\R)/(x^4) \, .$$
Thus, a right inverse to the canonical projection $\mathfrak R \rightarrow \overline{\mathfrak R}$ cannot be defined. Consequently, the canonical short exact sequence of $\mathfrak R$ given in \eqref{eq:split1} does not split in this case, meaning $\mathfrak R$ is not a split $\C$-superring.
\end{example}

\section{$\C$-SuperGeometry}\label{sec:three}
This section  is dedicated to a deeper examination of the structure of $\C$-superspaces. We begin by introducing fundamental definitions. Specifically, we use the notion of $\C$-superring to define affine $\C$-superschemes. We then show that $\C$-superspaces satisfying certain gluing conditions can be patched together to form a new $\C$-superspace. Subsequently, we introduce the concepts of \textit{projected} and \textit{split} $\C$-superspace, which are direct generalizations of the analogous concepts defined for supermanifolds, as discussed in \cite{Noja}.

\begin{definition}
A \emph{$\C$-superspace} is a locally $\C$-superringed-space, that is, a pair $\mathscr X=(X,\sh_{\mathscr X})$ consisting of a topological space $X$ together with a sheaf $\sh_{\mathscr X}$ of $\C$-superrings on $X$. For any $x\in X$, the stalk $\sh_{\mathscr{X}, x}$ is a local $\C$-superring given by
$$
\sh_{\mathscr{X},x}:=\varinjlim_{x\in U}\sh_{\mathscr X}(U),
$$
where the direct limit is taken over all open neighborhoods $U$ of $x$.
\end{definition}

For any $\C$-superring we define a \textit{super-spectrum functor} $\text{sSpec}(\cdot)$ which assigns to any $\C$-superring a $\C$-superspace. Conversely, we define a \textit{super global sections functor} $s\Gamma(\cdot)$ that assigns to any $\C$-superspace a $\C$-superring. More precisely,

\begin{itemize}[leftmargin=*]
\item {\bf The Superspectrum Functor:}  $\text{sSpec}:  {\bf S\C Rings}^{op} \rightarrow \bf{LS\C Rings}$ is defined on objects by $\text{sSpec}(\mathfrak{R})=(X_\mathfrak{R},\mathcal{O}_{X_\mathfrak{R}})$, for any $\C$-superring $\mathfrak R$, and on morphisms by $\text{sSpec}(\varphi)=(f_{\varphi},f^{\#}_{\varphi})$, for any $\C$-superring morphism $\varphi: \mathfrak R \rightarrow \mathfrak S$. \\
Here $X_{\mathfrak{R}}$ is the topological space of all $\R$ points for $\mathfrak{R}\ev$, that is, the set of all morphisms of $\C$-rings $x:\mathfrak{R}\ev\to\R$ and the structure sheaf is defined for any open subset $U\subseteq \mathfrak{R}$ as $$\sh_X(U)=\sh_{X_{\mathfrak{R}\ev}}(U)\oplus \text{Mspec}\, \mathfrak{R}\od(U)$$

\item {\bf The Global Sections Functor:} $s\Gamma:{\bf LS\C Rings}\rightarrow {\bf S\C Rings}^{op}$ is defined on objects by the global sections $s\Gamma(X,\mathcal{O}_{\mathscr X})=\mathcal{O}_{\mathscr X}(X)$, for any locally $\C$-superspace $(X, \mathcal O_{\mathscr X})$,  and on morphisms by $s\Gamma(f)=f^{\#}(X):\mathcal{O}_{\mathfrak Y}(Y)\rightarrow\mathcal{O}_{\mathscr X}(X)$, for any morphism of locally $\C$-superspaces $(f, f^{\#}) : (X, \mathcal O_{\mathscr X}) \rightarrow (Y, \mathcal O_{\mathfrak Y})$.
\end{itemize}

By a \textit{fair} $\C$-superring $\mathfrak{R}$ we mean a finitely generated $\C$-superring such that, for all $a \in \mathfrak R$ and $x \in X_{\mathfrak R}$, $\mathcal{L}_x(a)=0$ if and only if $a=0$, with $\mathcal{L}_x:\mathfrak{R}\rightarrow(\mathfrak{R}\ev)_{x}\oplus(\mathfrak{R}\od)_{x}$ the localization morphism. \\
For example $\C$ superrings associated to supermanifolds are  fair $\C$-superrings. Moreover, the category of supermanifolds forms a full (proper) subcategory of the category of fair affine $\C$-superschemes.\\
An important consequence in \cite{ORTG} is that the functor $\text{sSpec}$ is right adjoint to the functor $s\Gamma$. Furthermore, the natural transformation  $\mathcal{F}a \Rightarrow s\Gamma\circ \text{sSpec}$ is a natural isomorphism of functors, where $\mathcal{F}a$ is the \textit{fairfication functor}. More details in \cite[Theorem 4.34]{ORTG} and \cite[Proposition 4.36]{ORTG}. 
\begin{definition}
 An affine  $\C$-superscheme is a $\C$-superspace that is isomorphic to $\text{sSpec}(\mathfrak R)$ for some $\C$-superring $\mathfrak R$.
\end{definition}

\subsection{$C^\infty$-superspaces from $C^\infty$-spaces by ``adjoining" odd sections}

Starting from the $\C$-ring of global sections of a smooth manifold, we can construct the structure sheaf of a supermanifold as the tensor product of this $\C$-ring with the exterior algebra of a module over this ring. This process can be generalized to a broader class of algebraic objects known as Fermat theories. For example, in \cite{Carchedi}, the authors define a superization functor that associates a super Fermat theory with any given Fermat theory. The following is an example of this type of construction.

\begin{example}
  Let $\mathfrak{C}$ be the $\C$-ring of global sections of a smooth manifold $M$. We define the $\C$-superring $\mathfrak R$ as  $\mathfrak{C}[\theta^1,\ldots,\theta^q]^\pm$, where $\theta^1,\ldots,\theta^q$ are odd variables. The super spectrum of $\mathfrak R$ is the $\C$-superspace $(X_\mathfrak{C}, \mathcal{O}_{\mathscr X})$, where $X_\mathfrak{C}$ is the space of $\R$-points of $\mathfrak{C}$ and $\mathcal{O}_{\mathscr X}$ is the structure sheaf of $\C$-superrings whose ring of global sections is $\mathfrak R$. This example can be generalized by replacing the $\C$-ring of global sections of a smooth manifold with the $\C$-ring of global sections of an affine $\C$-scheme that is not a smooth manifold.
\end{example}

The procedure described is a well-known construction in supermanifold theory. Given a manifold $(M,\C_M)$ and a vector bundle $E$ over it, we can construct a supermanifold $\mathscr M=(M, \C(M)\otimes \bigwedge E^*)$. This supermanifold is completely characterized by the pair $(M,E)$. The basis vectors of the dual bundle $E^*$ can be interpreted as anticommuting coordinates. This ``bottom-up" approach for constructing specific supermanifolds can be generalized to $\C$-superspaces. We begin with an affine $\C$-scheme $(X,\sh_X)$ and a vector bundle $\mathscr E$ of rank $q$ over $X$. This vector bundle is an $\sh_X$-module such that for any open subset $U \subseteq X$, the restriction $\mathscr E|_U$ is isomorphic to $\mathscr E|_U\cong \sh_X|_U\otimes_{\R} \R^n$. From this, we can define a $\C$-superspace $\mathscr X=(X,\sh_X\otimes\bigwedge \mathscr E^*)$.

\begin{remark}
A specific example of this construction is detailed in \cite[Example 1.4]{Yau1}. In this work, the authors begin with a smooth manifold $(M,\sh_M)$  that has a Riemannian structure, along with a spinor bundle $S$ over it. By tensoring with the exterior bundle of $S^*$, they define a $\C$-superscheme $(M,\bigwedge_{\sh_M} S^*)$.
\end{remark}
\begin{example}\label{supchemes}
We can build a $\C$-superspace out of a $\C$-ringed space by superizing its structure sheaf. Indeed, if $(X,\sh_X)$ is a $\C$-ringed space, then for any $q\geq 0$ we define a sheaf of $\C$-superrings $\sh_{\mathscr X}$ by 
$$\sh_{\mathscr X}(U)=\sh_X(U)[\theta^1\ldots,\theta^q]^{\pm}.$$
With the obvious restrictions. It is easy to see that $\mathscr X =(X, \sh_{\mathscr X})$ is a $\C$-superspace. We call this space \emph{the superization of $(X,\sh_X)$ to $q$ odd variables}. 

An important result related to this construction is the localization at an $\R$-point. More precisely, let $\mathfrak{C}$ be a $\C$-ring and consider the $\C$-superring $\mathfrak{R}=\mathfrak{C}[\theta^1,\ldots,\theta^n]^\pm$. Let $x$ be an $\R$-point of $\mathfrak R$. The reduced part of $\mathfrak R$ corresponds to $\mathfrak{C}$, and its localization at $x$ is given by $\mathfrak{R}_x=\mathfrak{C}_x[\theta^1,\ldots,\theta^n]^\pm$. More details in \cite[Example 4.16]{ORTG}
\end{example}

We aim to treat both singular and regular spaces on an equal footing. For this purpose, the notion of a \emph{differentiable space} provides the appropriate framework. Differentiable spaces offer a natural enlargement of the category of smooth manifolds, designed to accommodate singular spaces while retaining a sheaf-theoretic notion of smooth functions. Formally, they are locally ringed spaces modeled on the zero loci of smooth functions in Euclidean space. In this way, they generalize the classical extrinsic construction of manifolds. It's important to note that the definition of a differentiable space is not unique in the literature. Various frameworks, including those based on differentiable algebras, sheaf-theoretic methods, and categorical approaches, offer distinct yet related definitions, all of which seek to extend smooth manifolds to more general spaces while preserving a differentiable structure. For further details, see \cite{Ler1}, \cite{GS}, and the references therein.\\

With the notion of differentiable space at hand, the following example provides a straightforward method for constructing $\C$-superspaces that are not supermanifolds.

\begin{example}\label{difsuperspace}
 Let $(X,\mathcal{F})$ be a differentiable space as defined in \cite{Ler1}. $X$ is not necessarily an affine $\C$-scheme coming from a smooth manifold, however $\mathcal F$ is a sheaf of $\C$-rings. We can therefore ``superize" $\mathcal F$ by adjoining a fixed number of odd variables to make it a sheaf of $\C$-superrings. Specifically, for a fixed $q\in \mathbb{N}$ and for any open subset $U$ of $X$, we define:

$$\sh_{\mathscr X} (U)=\mathcal{F}(U)[\theta^1,\ldots,\theta^q]^{\pm}$$

The restrictions are defined naturally by superizing those of $\mathcal F$. Thus, $\sh_{\mathscr X} $ is a sheaf of $\C$-superrings, and ${\mathscr X}=(X,\sh_{\mathscr X})$ is a $\C$-superspace, but is not necessarily  a supermanifold. 
\end{example}

Now that we have established a concrete procedure for constructing relatively simple $\C$-superspaces from $\C$-ringed spaces by adjoining odd variables, we want to be able to patch these superspaces together to obtain more intricate spaces. For this process to be successful, we need to specify certain gluing conditions.

\begin{definition}[Gluing data]
 Let $\{(X_i,\sh _i)\}_{i \in \mathcal I}$ be a collection of $\C$-superspaces and $\varphi_{ij}:(X_{ij}, \sh_{ij})\to (X_{ji},\sh_{ji})$ be a set of isomorphisms, where $X_{ij}=X_i\cap X_j$ and $\sh_{ij}=\sh_{i}|_{X_{ij}}$. We say that $\{(X_i,\sh _i)\}$ and the isomorphisms $\varphi_{ij}$ form a set of gluing data if the following cocycle conditions are satisfied:
 \begin{itemize}
     \item $\varphi_{ii}=Id$, 
     \item $\varphi_{ij}=\varphi_{ji}^{-1}$,
     \item $\varphi_{ij}\circ \varphi_{jk}=\varphi_{ik}$.
 \end{itemize}

\end{definition}

\begin{definition}
Let $\mathscr X=(X,\sh_{\mathscr X})$ be a $\C$-superspace. An \emph{open subspace} of $\mathscr X$ is a $\C$-superspace $\mathscr X|_U:=(U,\sh_{\mathscr X}|_U)$, where $U$ is an open subset of $X$ and $\sh_{\mathscr X}|_U$ is the restriction sheaf of $\sh_{\mathscr X}$ to $U$. This structure is equipped with a morphism $(j,j^\sharp):\mathscr X|_U\to \mathscr X$, where $j: U \to X$ is the inclusion map and the sheaf morphism $j^\sharp:j^{-1}\sh_{\mathscr X}\to\sh_{\mathscr X}|_U$ is the identity. The morphism $(j,j^\sharp)$ is referred to as the \emph{inclusion} of $ \mathscr X|_U$ in $\mathscr X$. We will systematically identify open subspaces with the underlying open sets. This allows for the meaningful definition of unions and finite intersections of open subspaces.
\end{definition}

The proofs of the following proposition and lemma are a direct extension of the non-graded case, which were established in \cite{Olarte}.

\begin{proposition}
Let $\{(X_i,\sh _i)\}_{i \in I}$ and $\varphi_{ij}$  be a set of gluing data. Then there exists a $\C$-superspace $(X,\sh_{\mathscr X})$, together with open immersions $\psi_{i}:(X_i,\sh_{X_i})\to (X,\sh_{\mathscr X})$ for each $i \in \mathcal I$. 
\end{proposition}

\begin{lemma}[Gluing of morphisms]\label{gluing}
Let $\{U_\alpha\}_{\alpha \in I}$ be an open cover of a $\C$-superspace $\mathscr X=(X,\sh_{\mathscr X})$, and let $\mathscr Y=(Y,\sh_{\mathscr Y})$ be a $\C$-superspace. Suppose that for each $\alpha \in I$ there is a morphism
$$
(f_\alpha, f_\alpha^\sharp): (U_\alpha, \sh_{\mathscr X}|_{U_\alpha}) \longrightarrow \mathscr Y
$$
such that, for all $\alpha, \beta \in I$,
$$
f_\alpha|_{U_{\alpha\beta}} = f_\beta|_{U_{\alpha\beta}},
$$
where $U_{\alpha\beta} := U_\alpha \cap U_\beta$. Then there exists a unique morphism $(f,f^\sharp): \mathscr X \to \mathscr Y$ such that $f|_{U_\alpha}=f_\alpha$ for every $\alpha \in I$.
\end{lemma}
With the structures defined above, we can now provide a formal definition of a $\C$-superscheme.

\begin{definition}
A $\C$-superscheme $\mathscr X=(X, \sh_{\mathscr X})$ is a $\C$ superspace that is locally affine. This means that for any $x\in X$, there exists an open neighborhood $U $ of $x$ such that $\mathscr X|_U=(U, \sh_{\mathscr X}|_U)$ is isomorphic to $\text{sSpec}(\mathfrak{R})$ for some $\C$-superring $\mathfrak{R}$.
\end{definition}

\subsection{Derivations}
In the following, we present some basic definitions concerning the theory of derivations for $\C$-superrings. These concepts are consistent with the established theory of derivations for superrings. A natural refinement of this notion is to require compatibility with the underlying $C^\infty$-structure, imposing that an even superderivation of a $\C$-superring induce a genuine $C^\infty$-derivation on its reduced ring.

\begin{definition} \label{def:superderivation}
An $\R$-\emph{derivation} of a $\C$-superring $\mathfrak{R}$ with values in $\mathfrak{R}$ is a homogeneous ${\R}$-linear map $D:\mathfrak{R}\rightarrow\mathfrak{R}$ that satisfies the $\Z_2$-graded Leibniz rule:
  $$
    D(rs)= D( r)s\,+\, (-1)^{|D||r|}rD (s)
  $$
for all homogeneous elements $r$ and $s$  in $\mathfrak{R}$. We say that $D$ is an even derivation if $|D|=0$ and an odd derivation if $|D|=1$.
  \end{definition}

 The set $sDer_{\mathfrak{R}}$ of all even and odd derivations on a $\C$-superring $\mathfrak R$ forms a left $\mathfrak{R}$-module. The module action is defined by $(rD)(s):= r D(s)$. Here, the product on the right side of the equality is the product of the ring $\mathfrak{R}$. The parity of the new derivation $rD$ is given by $|rD|:= |r|+|D|$ for any $r\in \mathfrak{R}$ and $D\in sDer_{\mathfrak{R}}$.

\begin{definition}\label{def:Csuperderivation}
Let $\mathfrak{R}=\mathfrak{R}\ev\oplus \mathfrak{R}\od$ be a $\C$-superring, and let $\overline{\mathfrak{R}}$ denote its reduced $C^\infty$-ring. A $C^\infty$-superderivation of $\mathfrak{R}$ is an even $\R$-derivation such that $D$ preserves the ideal defining the reduced structure and induces a $C^\infty$-derivation
$\overline{D}:\overline{\mathfrak{R}}\longrightarrow \overline{\mathfrak{R}}$. Equivalently, for every $f\in C^\infty(\mathbb{R}^n)$ and every $c^1,\ldots,c^n\in \overline{\mathfrak{R}}$, $\overline{D}$ satisfies:
\[
\overline{D}
\bigl(
\Phi_f(c^1,\ldots,c^n)
\bigr)
=
\sum_{i=1}^{n}
\Phi_{\partial_i f}
(c^1,\ldots,c^n)
\overline{D}(c^i).
\]
\end{definition}
 
\begin{remark}
A few explanatory notes will help contextualize these two previous definitions. Firstly, the restriction to even superderivations is motivated by the behavior of the reduced structure upon passing to the quotient. Indeed, an odd superderivation maps the even part of a superring into its odd part, and thus does not, in general, induce a derivation on the reduced ring (which is a purely even, ordinary ring). Consequently, the notion of a $\C$-superderivation is formulated here exclusively for even superderivations. Furthermore, we require compatibility with the $\C$-structure only on the reduced ring, rather than on the entire even part $\mathfrak{R}\ev$. This is a technical simplification: the even nilpotent elements of $\mathfrak{R}\ev$, together with the odd elements of $\mathfrak{R}$, are viewed as belonging to the auxiliary superalgebraic structure and are therefore governed only by the graded Leibniz rule. This approach avoids imposing an independent $\C$-ring structure on the full even part while preserving the smooth structure on its reduction.
\end{remark}

\begin{example} 
Let $\mathfrak{R}=\mathfrak{C}[\theta^1,\ldots,\theta^q]^\pm$ be a split $\C$-superring with $\{\phi_f\}$ the collection of $n$-ary operations and $\mathfrak{C}$ a finitely generated reduced $\C$-ring. If the generators of $\mathfrak{C}$ are $c^1,\ldots, c^p$, then for each $c^i$ we have an even derivation, $\frac{\partial}{\partial c^i}:\mathfrak{C}\longrightarrow\mathfrak{C}$, defined by 
$\frac{\partial}{\partial c^i}(c^j)=\delta_i^j$, where $\delta_i^j$ is the Kronecker delta.
Therefore, as a consequence of the chain rule, we have 
$$
\frac{\partial}{\partial c^i}(\phi_f(c^1,\ldots,c^n))=\phi_{\partial_if}(c^1,\ldots,c^n)
$$
for any $f\in \C(\R^n)$.\\
This derivation can be extended linearly to $\mathfrak R$ by defining $\frac{\partial}{\partial c^i}(c\theta^j)=\frac{\partial}{\partial c^i}(c)\theta^{j}$ for any $c\in\mathfrak{C}$ and any odd variable $\theta^j$.

We can also define the odd derivations, $\frac{\partial}{\partial \theta^i}:\mathfrak{R}\longrightarrow\mathfrak{R}$, associated with the odd variables $\theta^i$. These derivations are formally defined for all $i,j$ by the relations:
 $$
 \frac{\partial}{\partial \theta^i}(\theta^{j})=\delta_i^j\quad\text{and}\quad\frac{\partial}{\partial \theta^i}(c^{j})=0
 $$ 

It is not hard to show that the set $\{\frac{\partial}{\partial c^i}, \frac{\partial}{\partial \theta^i}\}$ forms a basis for the $\mathfrak R$-module $sDer_\mathfrak{R}$.
 
\end{example}

Now that we've explained the notion of superderivations, we can globalize this concept and define the tangent sheaf of a $\C$-superspace.

\begin{definition} 
Let $\mathscr X=(X,\sh_{\mathscr X} )$ be a $\C$-superspace. The tangent sheaf of $\mathscr X$ is the $\sh_{\mathscr X}$-module $\mathcal T_{\mathscr X}:=sDer_{\sh_{\mathscr X}}$, which is defined on any open subset $U\subseteq X$ as 

$$\mathcal T_{\mathscr X}(U):=sDer_{\sh_{\mathscr X}(U)} \, .$$
\end{definition}

\begin{remark}
Given an $\sh_{\mathscr X}$-module $\mathcal{G}$, we can define superderivations with values in $\mathcal{G}$. These are maps $\sh_{\mathscr X}\rightarrow \mathcal{G}$ that satisfy the $\Z_2$-graded Leibniz rule. Such $\sh_{\mathscr X}$-derivations can be represented by the sheaf $\mathcal T_{\mathscr X}\otimes \mathcal{G}$.
\end{remark}

\subsection{Structure theorems}

The category of $\C$-superspaces is significantly more extensive than that of smooth or complex supermanifolds. However, a straightforward method exists to construct supermanifold examples from classical non-graded geometry. 

Given a smooth manifold $M$ with its structure sheaf $\mathcal O_M$, and a locally free sheaf $\mathcal E$ of finite rank over $\mathcal O_M$ (which corresponds to a vector bundle $E$ over $M$), we can define a supermanifold. This supermanifold has the topological space $M$ and a structure sheaf given by the exterior algebra $\Lambda \mathcal E$. Such supermanifolds are called \textit{split}. Both smooth and complex supermanifolds are, by definition, locally split. A key question is whether all supermanifolds are globally split.

Batchelor's theorem \cite{Batchelor} addresses this question in the real smooth case, establishing that every smooth supermanifold is globally split. However, this theorem does not hold for complex supermanifolds due to topological obstructions (see, e.g., \cite{CNR}).

In this section, we will introduce the concepts and analyze the necessary conditions and potential obstructions to extend Batchelor's theorem to a class of $\C$-superspaces that properly includes smooth supermanifolds.

\begin{definition}
Let $\mathscr X=(X\mathscr,\mathcal{O}_{\mathscr X})$ be a $\C$-superspace. The \emph{canonical superideal} of $\mathscr X$ is the $\mathcal{O}_{\mathscr X}$-module $\mathcal{J}_{\mathscr X}$ generated by odd section,defined for any open $U\subseteq X$ as: 
$$
\mathcal{J}_{\mathscr X}(U)=\mathcal{O}_{\mathscr X}(U) \; (\mathcal{O}_{\mathscr X}(U))\od =[(\mathcal{O}_{\mathscr X}(U))\od]^2\oplus (\mathcal{O}_{\mathscr X}(U))\od.
$$  
\end{definition}

\begin{remark}
The canonical superideal $\mathcal{J}_{\mathscr X}$ could be used to assign an ordinary $\C$-ringed space to any $\C$-superspace $\mathscr X=(X,\mathcal{O}_{\mathscr X})$. Indeed, the reduced $\C$-ringed space associated to $\mathscr X$ is given by $\mathscr X_{red}=(X,\sh_{\mathscr X_{red}})$, where 
$$
\sh_{\mathscr X_{red}}(U)=\overline{\sh_{\mathscr X}(U)}=\sh_{\mathscr X}(U)/\mathcal{J}_{\mathscr X}(U),
$$
is the superreduced space of the $\C$-superring $\sh_{\mathscr X}(U)$ for any open $U\subseteq X$.
\end{remark}

Every $\C$-superspace $\mathscr X=(X, \sh_{\mathscr X})$ comes equipped with a closed immersion $(j,j^{\sharp}):\mathscr X_{red}\longrightarrow \mathscr X$. Here, $j:X \rightarrow X$ is the identity, and $j^{\sharp}:\sh_{\mathscr X} \longrightarrow \sh_{\mathscr X_{red}}$ is the corresponding quotient projection.

Another property of the canonical ideal is that it induces a \emph{$\mathcal{J}_{\mathscr X}$-adic filtration:} 
\begin{equation} \label{Jfiltration}
\mathcal{O}_{\mathscr X}=\mathcal{J}_{\mathscr X}^0\supset \mathcal{J}_{\mathscr X} \supset \mathcal{J}_{\mathscr X}^2\supset\cdots \supset \mathcal{J}_{\mathscr X}^q \supset \mathcal{J}_{\mathscr X}^{q+1}\cdots .
\end{equation}
This filtration is finite if there exists $q$ such that annihilates $\mathcal{J}_{\mathscr X}$, that is, $\mathcal{J}_{\mathscr X}^{q+1}=0$.

\begin{definition}
Let $\mathscr X=(X,\mathcal{O}_{\mathscr X})$ be a $\C$-superspace, and let $\mathcal{J}_{\mathscr X}$ be the canonical superideal.  With $\text{Gr}^{(k)}\, \mathcal O_{\mathscr X}=\mathcal{J}_{\mathscr X}^k/\mathcal{J}_{\mathscr X}^{k+1} $, the direct sum $\text{Gr}\, \mathcal\,{O}_{\mathscr X}=\bigoplus_{k\geq0}\text{Gr}^{(k)}\, {O}_{\mathscr X}$ is a sheaf of $\mathcal{O}_{\mathscr X}$-modules. The pair  $(X,\text{Gr}\, \mathcal{O}_{\mathscr X}):=\text{Gr}\, \mathscr X$ is called the \emph{graded $\C$-superspace associated to} $\mathscr X$.
\end{definition}

\begin{remark}\label{EquiSymGr}
Let $\mathscr X=(X,\sh_{\mathscr X})$ be a $\C$-superspace and let $\mathcal F_{\mathscr X}=\mathcal J_{\mathscr X}/\mathcal J_{\mathscr X}^2$ be its associated fermionic sheaf. The symmetric powers of $\mathcal{F}_{\mathscr{X}}$, denoted by $\text{Sym}^n \mathcal{F}_{\mathscr{X}}$, are identified with the quotients of the $\mathcal{J}_{\mathscr{X}}$-adic filtration as $\mathcal{O}_{\mathscr{X}_{\text{red}}}$-modules:$$\text{Sym}^n \mathcal{F}_{\mathscr{X}} \cong \mathcal{J}^n_{\mathscr{X}} / \mathcal{J}^{n+1}_{\mathscr{X}}.$$Locally, the $n$-th symmetric power of $\mathcal{F}_{\mathscr{X}}$ is generated by products of $n$ odd sections, which can be expressed in the form $\sum_{|I|=n} c^I \theta^I$, where $I$ is a multi-index of length $n$ and $c^I \in \mathcal{O}_{\mathscr{X}_{\text{red}}}$. It follows that there is an isomorphism of sheaves of graded $\mathcal{O}_{\mathscr{X}_{\text{red}}}$-algebras:$$\text{Sym} \, \mathcal{F}_{\mathscr{X}} \cong \text{Gr} \, \mathcal{O}_{\mathscr{X}}.$$
\end{remark}

\begin{definition}
We say that $\mathscr X$ has \emph{finite odd dimension} if $\mathcal F_{\mathscr X}$ is a locally free $\sh_{\mathscr X_{red}}$-module of finite rank, that is, for every $x\in X$ there exists an open neighborhood $U$ of $x$, $q\in\mathbb N$, independent of $U$, and local generators $\theta^1,\ldots,\theta^q\in \mathcal F_{\mathscr X}(U)$ such that
$$
\mathcal F_{\mathscr X}(U) \cong \bigoplus_{i=1}^q \sh_{\mathscr X_{red}}(U)\cdot\theta^i
$$
as $\sh_{\mathscr X_{red}}(U)$-modules.  We emphasize that this finiteness refers exclusively to the fermionic sheaf $\mathcal F_{\mathscr X}$, viewed as an $\sh_{\mathscr X_{red}}$-module: it records the local rank of the quotient $\mathcal J_{\mathscr X}/\mathcal J_{\mathscr X}^2$, and nothing more. In particular, it has no bearing on any finiteness property of the reduced $\C$-ring $\sh_{\mathscr X_{red}}(U)$ itself; having finite odd dimension neither requires nor implies that $\sh_{\mathscr X_{red}}(U)$ be finitely generated, fair, or otherwise finite as a $\C$-ring.
\end{definition}

\begin{definition}\label{def:severaldef}
A $\C$-superspace $\mathscr X=(X,\sh_{\mathscr X})$ is 
\begin{itemize}
\item \emph{split} if it is isomorphic to its associated graded $\C$-superspace $\text{Gr}\, \mathscr X$.
    \item \emph{locally split } if for every $x\in X$ there exists an open subset $U\subseteq X$ containing $x$ such that $(U,\sh_{\mathscr X}|_U)$ is a split affine $\C$-superspace.
    \item \emph{locally split of rank $q$} if $\mathcal{O}_{\mathscr X}$ is locally of the form $\mathfrak{C}_U[\theta^1,\ldots,\theta^q]^\pm$ for some $\C$-ring $\mathfrak{C}_U$.
    \item \emph{locally finetely generated} if it is locally of finite rank $q$ and $\mathfrak{C}_U$ is finitely generated as $\C$-ring.
\end{itemize}
\end{definition}
\begin{remark}
The notions of locally split of rank $q$ and finite odd dimension introduced above are logically independent, and in fact only one implication holds in general:
$$
\mathscr X \text{ locally split of rank } q \;\Longrightarrow\; \mathscr X \text{ has odd dimension } q \, .
$$
Indeed, if $\sh_{\mathscr X}(U)\cong \mathfrak C_U[\theta^1,\ldots,\theta^q]^{\pm}$ on some open $U$, then $\mathcal F_{\mathscr X}(U)=\mathcal J_{\mathscr X}(U)/\mathcal J_{\mathscr X}(U)^2$ is freely generated by the classes of $\theta^1,\ldots,\theta^q$, hence locally free of rank $q$.\\
The converse fails: having finite odd dimension is a first-order condition, depending only on the quotient $\mathcal J_{\mathscr X}/\mathcal J_{\mathscr X}^2$, whereas local splitness requires an isomorphism of the entire structure sheaf $\sh_{\mathscr X}$ with the exterior algebra model, not merely of its associated graded pieces. By Remark \ref{EquiSymGr}, the identification $\text{Sym}^n \mathcal F_{\mathscr X} \cong \mathcal J_{\mathscr X}^n/\mathcal J_{\mathscr X}^{n+1}$ holds regardless of whether $\mathscr X$ is locally split; finite odd dimension only guarantees that each graded piece of the $\mathcal J_{\mathscr X}$-adic filtration is locally free of finite rank, but says nothing about how these pieces are extended to reconstruct $\sh_{\mathscr X}$ itself.

This is witnessed by the $\C$-superring $\mathfrak{R}=C^\infty(\R^{1|2})/(x^2+\theta^1\theta^2)$ of Example \ref{nonsplit}. The relation $x^2+\theta^1\theta^2$ lies entirely in the even part and does not affect $\theta^1,\theta^2$ as generators of $\mathfrak R\od$, so $\mathfrak J_{\mathfrak R}/\mathfrak J_{\mathfrak R}^2$ remains free of rank $2$ over $\overline{\mathfrak R}=C^\infty(\R)/(x^2)$; thus $\mathscr X=\text{sSpec}(\mathfrak R)$ has odd dimension $2$. Moreover $\mathfrak R$ is fair, since its reduced part $C^\infty(\R)/(x^2)$ is fair. However, $\mathfrak R$ is not split. Hence $\mathscr X$ has finite odd dimension and is fair, yet it is not locally split.
\end{remark}

\begin{remark}\label{rem:whyfair}
At this point we have two, a priori independent, notions of ``split'': an \emph{algebraic} one, for a $\C$-superring $\mathfrak R$ (Definition~\ref{def:Csplit}), asserting that $\mathfrak R$ is isomorphic to $\mathfrak C[\theta^1,\ldots,\theta^q]^{\pm}$ for some $\C$-ring $\mathfrak C$; and a \emph{geometric} one, for a $\C$-superspace $\mathscr X$, asserting that $\mathscr X$ is isomorphic to its associated graded superspace $\mathrm{Gr}\mathscr X$. These two notions are related by the superspectrum and global sections functors, but the relation is not automatic: knowing that $\mathrm{sSpec}(\mathfrak R)$ is split as a $\C$-superspace does not, by itself, tell us that $\mathfrak R$ is split as a $\C$-superring, and conversely.

The obstruction lies entirely in the functorial relationship between $\mathfrak R$ and $s\Gamma(\mathrm{sSpec}(\mathfrak R))$. In general there is only a natural transformation
$$
\mathfrak R \longrightarrow s\Gamma(\mathrm{sSpec}(\mathfrak R)) \, ,
$$
namely the fairification map, and this map need not be an isomorphism: $s\Gamma(\mathrm{sSpec}(\mathfrak R))$ recovers only the \emph{fairification} of $\mathfrak R$, which can differ from $\mathfrak R$ itself when $\mathfrak R$ fails to be germ-determined. Consequently, an isomorphism of $\C$-superspaces $\mathrm{sSpec}(\mathfrak R)\cong (X,\mathrm{Gr}\,\sh_{\mathrm{sSpec}(\mathfrak R)})$, upon taking global sections, yields information about $s\Gamma(\mathrm{sSpec}(\mathfrak R))$ rather than about $\mathfrak R$ directly, and the two may not agree.

This tension is resolved precisely when $\mathfrak R$ is \emph{fair}: by \cite[Theorem~4.34, Proposition~4.36]{ORTG}, the natural transformation $\mathcal Fa\Rightarrow s\Gamma\circ \mathrm{sSpec}$ is an isomorphism on fair $\C$-superrings, so that $s\Gamma(\mathrm{sSpec}(\mathfrak R))\cong \mathfrak R$ whenever $\mathfrak R$ is fair. Under this hypothesis, taking global sections becomes a faithful operation on $\mathfrak R$ itself, and the algebraic and geometric notions of splitness can be legitimately identified: a splitting of the $\C$-superspace $\mathrm{sSpec}(\mathfrak R)$ corresponds, via global sections, to an actual splitting of the $\C$-superring $\mathfrak R$, and not merely of its fairification.

This is the reason we restrict attention to fair (and, in the sheaf-theoretic setting, locally fair) $\C$-superspaces throughout the remainder of this section: fairness is precisely the condition under which the algebraic theory of split $\C$-superrings developed in Section~\ref{sec:two} faithfully controls the geometric theory of split $\C$-superspaces. Without it, the equivalence stated in Theorem~\ref{thm:fair+split} below would fail already at the level of affine $\C$-superschemes, since $\mathfrak R$ could be non-split while its fairification is split, or vice versa.
\end{remark}
\begin{theorem}\label{thm:fair+split}
Let $\mathscr X=\text{sSpec}(\mathfrak{R})$ be a fair affine  $\C$-superscheme. $\mathscr X$ is split if and only if $\mathfrak{R}$ is split. 
\end{theorem}

\begin{proof}
  Let $\mathscr X=(X,\sh_{\mathscr X})=\text{sSpec}(\mathfrak{R})$ be a split $\C$-superscheme. By definition, $\mathscr X$ is isomorphic to $(X,\text{Gr}\, \sh_{\mathscr X})$. Taking global sections, we obtain an isomorphism of $\C$-superrings:
  
  $$s\Gamma(X)\cong \text{Gr}\, \sh_{\mathscr X}(X) \, .$$

Observe that $\text{Gr}\, \sh_{\mathscr X}(X)$ is a split $\C$-superring. Specifically, $\text{Gr}\, \sh_{\mathscr X}(X)$ is isomorphic to the superring $\sh_{\mathscr X_{red}}(X)[\theta^1,\ldots,\theta^q]^{\pm}$. Given that $\mathfrak{R}$ is fair, we have an isomorphism $s\Gamma(\text{sSpec}(\mathfrak{R}))\cong \mathfrak{R}$, by \cite[Proposition 4.36]{ORTG}. Therefore, $\mathfrak{R}\cong \sh_{\mathscr X_{red}}(X)[\theta^1,\ldots,\theta^q]^{\pm}$, which implies that $\mathfrak{R}$ is split. \\

Conversely, if $\mathfrak{R}$ is split, then it is of the form $\mathfrak{C}[\theta^1,\ldots,\theta^q]^{\pm}$ for some $\C$-ring $\mathfrak{C}$. Then for any open subset $U\subseteq X$ we have $\sh_{\mathscr X}(U)=\sh_{\mathfrak{C}}(U)[\theta^1,\ldots,\theta^q]^{\pm}\cong Gr \, \sh_{\mathscr X}(U)$ (where $\sh_\mathfrak{C}$ is the structure sheaf of $\text{Spec}(\mathfrak{C})$) which allows us to conclude that $\mathscr X$ is split.
\end{proof}

\begin{remark}
If $\mathscr X$ is locally fair and locally split, then for any $x\in X$, there exists an open subset $U\subseteq X$ such that the ring of sections $\sh_{\mathscr X}(U)$ is a split $\C$-superring.
\end{remark}

We have a natural closed embedding $(Id,\pi) : \mathscr X_{red} \rightarrow \mathscr X$, which acts as the identity on the underlying topological space $X$, together with an associated morphism of sheaves:
$$
\pi:\sh_{\mathscr X}\rightarrow \sh_{\mathscr X_{red}}=\mathcal O_{\mathscr X}/\mathcal J_{\mathscr X}.
$$
This morphism is induced by the projection to the quotient, yielding the following \emph{structural short exact sequence} of $\mathscr X$:
\begin{equation}\label{sses}
0\longrightarrow \mathcal{J}_{\mathscr X} \stackrel{\iota} \longrightarrow \sh_{\mathscr X}\stackrel{\pi}{\longrightarrow} \sh_{\mathscr X red}\longrightarrow 0 \, .
\end{equation}

\begin{definition}

A  $\C$-superspace $\mathscr X= (X,\mathcal O_{\mathscr X}) $ is \emph{projected} if  there exists a morphism of sheaves $\sigma:\sh_{\mathscr X_{red}}\longrightarrow\sh_{\mathscr X}$, such that $\pi \circ \sigma= \mathrm{id}_{\sh_{\mathscr X_{red}}}$.
In other words,its structural short exact sequence splits.
\end{definition}

\begin{proposition}
A split $\C$-superspace is projected. 
\end{proposition}
\begin{proof}
Let $\mathscr{X}=(X,\mathcal{O}_{\mathscr{X}})$ be a split
$C^\infty$-superspace. By Definition~\ref{def:severaldef}, there exists an isomorphism of locally ringed $C^\infty$-superspaces
$\varphi:\mathscr{X}\xrightarrow{\sim}
\bigl(X,\operatorname{Gr}\mathcal{O}_{\mathscr{X}}\bigr)$. Since
\[
\operatorname{Gr}^0\mathcal{O}_{\mathscr{X}}
=
\mathcal{O}_{\mathscr{X}}/\mathcal{J}_{\mathscr{X}}
=
\mathcal{O}_{\mathscr{X}_{\mathrm{red}}},
\]
the canonical inclusion of the degree-zero component induces a morphism
$\iota:\mathcal{O}_{\mathscr{X}_{\mathrm{red}}}
\hookrightarrow
\operatorname{Gr}\mathcal{O}_{\mathscr{X}}$. Transporting this morphism through the isomorphism $\varphi$, we obtain a
morphism $\sigma:\mathcal{O}_{\mathscr{X}_{\mathrm{red}}}
\longrightarrow\mathcal{O}_{\mathscr{X}}$. By construction, the composition with the canonical projection $\pi:\mathcal{O}_{\mathscr{X}}
\longrightarrow\mathcal{O}_{\mathscr{X}_{\mathrm{red}}}$
is the identity. Hence $\pi\circ\sigma
=\operatorname{id}_{\mathcal{O}_{\mathscr{X}_{\mathrm{red}}}}$,
and therefore the structural short exact sequence splits. Thus,
$\mathscr{X}$ is projected.
\end{proof}  

A consequence of the last theorem is that we have a way to exhibit examples of projected and non-projected $\C$-superspaces. 
Furthermore, we obtain the following sequences of inclusions of  $\C$-superspaces
$$\{\text{Split}\}\subset \{\text{projected}\}\subset \{\text{locally-split}\}$$.

\begin{example}\label{nonprojected}
Consider the $\C$-superring $\mathfrak{R}=C(\R^{1|2})/(x^2+ \theta^1\theta^2)$. As we proved in \ref{nonsplit} there is no morphism $\sigma:\overline{\mathfrak{R}}\rightarrow \mathfrak{R}$  that is right inverse to the canonical projection $\pi:\mathfrak{R}\rightarrow \overline{\mathfrak{R}} $ equivalently, the exact sequence \begin{equation*} 
0\longrightarrow \mathfrak{J}_{\mathfrak{R}} \stackrel{\iota} \longrightarrow \mathfrak{R}\stackrel{\pi}{\longrightarrow} \overline{\mathfrak{R}}\longrightarrow 0
\end{equation*} does not split. 
On the other side, $\mathfrak{R}$ is a fair $\C$-superring since its reduced part $\C(\R)/(x^2)$ is fair.
Now for a $\C$-superspace to be nonprojected it is enough to show an open subset $U$ of $X$ such that the sequence \begin{equation*}
0\longrightarrow \mathcal{J}_{\mathscr X}(U) \stackrel{\iota} \longrightarrow \sh_{\mathscr X}(U)\stackrel{\pi}{\longrightarrow} \sh_{\mathscr X red}(U)\longrightarrow 0 \, .
\end{equation*}
does not split. Since $\mathfrak{R}$ is fair, we have that $\mathcal{O}_{\mathscr X}(X)\cong \mathfrak{R}$, thus taking $U=X$ we have that $\mathscr X$ is non projected.
\end{example}

\begin{definition}\label{spliting}
A \emph{splitting} of a $\C$-superspace $\mathscr X$ is an isomorphism $\mathscr X\to \text{Gr}\, \mathscr X$ that induces the identity on $\text{Gr}\, \mathscr X$.  
\end{definition}

Stated differently, a splitting of $\mathscr X= (X,\sh_{\mathscr X})$ is an isomorphism  $\sigma : \operatorname{Gr}\sh_{\mathscr X} \to\sh_{\mathscr X}$ such that for any $p$, where $0 \leq p \leq n$, the restriction $\sigma|_{\operatorname{Gr}^{(p)} \sh_{\mathscr X}}$ is a lift of the canonical morphism  $\mathcal{J}^p\to \operatorname{Gr}^{(p)}\sh_{\mathscr X}$ \cite[Section 3]{Koszul}. \\

In the smooth category, all supermanifolds are globally split, a property that fails to hold in the complex analytic category. This distinction is fundamentally rooted in the cohomology of sheaves. Specifically, for any fine sheaf $\mathcal{F}$ over a topological space, the higher cohomology groups vanish: $H^i(M, \mathcal{F}) = 0$ for all $i > 0$ \cite[Proposition 4.36]{Voisin}. A sheaf of modules over a sheaf of commutative rings $\mathcal{O}_M$ is defined as fine if $\mathcal{O}_M$ admits partitions of unity \cite[Proposition 4.35]{Voisin}. In the smooth setting, the existence of smooth partitions of unity ensures that the relevant sheaves are fine, thereby guaranteeing the triviality of the obstructions to global splitting. Conversely, Batchelor’s Theorem does not extend to the complex category because complex manifolds do not admit holomorphic partitions of unity. Consequently, the associated sheaves are not fine, allowing for non-vanishing cohomology groups that act as obstructions to a global split.\\

Given that $C^\infty$-superspaces constitute a category that properly contains smooth supermanifolds, it is natural to investigate whether global splitness is a property unique to the smooth supermanifold case. To delineate a broader class of $\C$-superspaces susceptible to global splitting, we isolate two sufficient conditions foundational to the proof of Batchelor’s Theorem (cf. \cite[Section 4.2]{Manin}):
\begin{enumerate}
\item Topological/Algebraic Condition: The structure sheaf must admit a partition of unity, ensuring that the sheaf is fine and that higher cohomology groups vanish.

\item Geometric Condition: The superspace must be locally split and possess a constant odd dimension.
\end{enumerate}

By formalizing these requirements, we arrive at the following definition:

\begin{definition}\label{BatchelorSpace}
A $\C$-superspace $\mathscr X= (X,\sh_{\mathscr X})$ is defined as a \emph{Batchelor space} if it is locally split and its structure sheaf $\sh_{\mathscr X}$ is fine. 
\end{definition}

According to Theorem 4.40 in \cite{J}, if a $C^\infty$-scheme $(X, \mathcal{O}_X)$ is defined over a Lindel\"of topological space $X$ with a smoothly generated topology, then the structure sheaf $\mathcal{O}_X$ is necessarily fine. This result establishes that a broad class of $C^\infty$-schemes possesses fine structure sheaves. Recall that smooth supermanifolds can be viewed as the superization of the sheaf of $C^\infty$-rings associated with a classical smooth manifold (cf. Example \ref{supchemes}). Consequently, to construct a Batchelor space that is not a supermanifold, one may superize a structure sheaf of $C^\infty$-rings that is both locally split and fine, yet does not originate from a smooth manifold (cf. Example \ref{difsuperspace}). It follows that the category of Batchelor spaces contains the category of smooth supermanifolds as a full, proper subcategory. We emphasize that the class of Batchelor spaces is strictly larger than that of smooth supermanifolds, providing a more general framework for global splitness.\\

A foundational result established by Joyce in \cite[ Theorem 4.20]{J} demonstrates a categorical equivalence between fair $C^\infty$-rings and fair affine $C^\infty$-schemes. This equivalence was subsequently extended to the supergeometric setting—specifically to $C^\infty$-superrings—in \cite{ORTG}. Given this correspondence, locally fair $C^\infty$-superspaces emerge as a highly robust and natural framework for developing smooth supergeometry (cf. \cite[Remark 4.24]{ORTG}). A central inquiry then arises: are locally fair $C^\infty$-superspaces necessarily Batchelor spaces? We provide an affirmative resolution to this question in the following proposition.

\begin{proposition}\label{prop:fairBat}
Let $\mathscr X = (X,\sh_{\mathscr X})$ be an affine fair and split $\C$-superspace. Then $\mathscr X$ is a Batchelor space.
\end{proposition}

\begin{proof}
The hypothesis that $\mathscr{X}$ is an affine fair and split $C^\infty$-superspace implies that $\mathscr{X}$ is isomorphic to the superspectrum of a fair split $C^\infty$-superring $\mathcal{R}$. By Proposition 4.20 of \cite{ORTG}, the reduced part of $\mathfrak {R}$ is a fair $C^\infty$-ring $\mathfrak C$. It follows that the underlying topological space $X$ is homeomorphic to the space of $\mathbb{R}$-points of $\mathfrak C$, which can be embedded as a closed subset of $\mathbb{R}^n$ for some $n \in \mathbb{N}$. As a closed subspace of Euclidean space, $X$ is necessarily Lindel\"of and possesses a smoothly generated topology. Consequently, by Theorem 4.20 in \cite{J}, the structure sheaf $\mathcal{O}_{\mathscr X}$ is fine. Given that $\mathscr {X}$ is split by assumption (and thus locally split), it satisfies the criteria for a Batchelor space.
\end{proof}

The proof establishing that every Batchelor space is globally split follows the construction for supermanifolds developed by Manin in \cite[Chapter 4. \S 2]{Manin} (see also \cite[Part I. Chapter 4. Sections 5-7]{Berezin}, \cite[Section 3.1]{Noja} and \cite[Section 2.2]{DonagiOtt}). Adopting Manin's approach, we first demonstrate that every Batchelor space is projected, subsequently showing that it is globally split. Throughout the following discourse, the tangent sheaf over the underlying topological space $X$ is denoted by $\mathcal{T}_X$. 

\begin{proposition}\label{projected}
Let $\mathscr X=(X,\sh_{\mathscr X})$ be a Batchelor superspace and let $\mathscr X_{red}=(X, \sh_{\mathscr X_{red}})$ be its associated reduced space, equipped with the canonical projection $\pi:\mathcal O_{\mathscr X} \to \mathcal O_{\mathscr X_{red}}$. Then $\mathscr X$ is projected; that is, there exist a morphism of $\C$-superspaces $(Id_X,\sigma)$, where $\sigma: \mathcal O_{\mathscr X_{red}} \to \mathcal O_{\mathscr X}$ is a morphism of sheaves of $C^\infty$-superrings acting as a right inverse to $\pi$, satisfying $ \pi\circ \sigma= Id_{\mathcal O_{\mathscr X_{red}}}$.    
\end{proposition}

\begin{proof}
The proof is formally analogous to the construction for supermanifolds developed by Manin in \cite[Chapter 4. \S 2. Lemma 4]{Manin}. The argument proceeds via an inductive lifting argument along the nilpotent filtration of the ideal $\mathcal J_{\mathscr X}$, subsequently employing the vanishing of sheaf cohomology for fine sheaves to patch local algebraic splittings into a global one. The primary conceptual distinction lies in the fact that the morphism $\sigma:\mathcal O_{\mathscr X_{red}}\longrightarrow \mathcal O_{\mathscr X}$ is a morphism of sheaves of $\C$-superrings.\\
Consider the collection of ringed spaces $\mathscr X^{(k)}=\left( X, \mathcal O^{(k)}_{\mathscr X} \right)$, where the structure sheaves are defined by the quotients $\mathcal O^{(k)}_{\mathscr X}:=\mathcal O_{\mathscr X}/\mathcal J_{\mathscr X}^k$. Note that for $k=1$, we have $\mathcal O_{\mathscr X}^{(1)}=O_{\mathscr X}/\mathcal J_{\mathscr X}=\mathcal O_{\mathscr X_{red}}$. If the odd dimension of $\mathscr X$ is $q$,  then $\mathcal J_{\mathscr X}^{q+1}=0$, implying $\mathcal O_{\mathscr X}^{(q+1)}=O_{\mathscr X}/\mathcal J_{\mathscr X}^{q+1}=\mathcal O_{\mathscr X}$. The $\mathcal J_{\mathscr X}$-adic filtration (Equation \ref{Jfiltration}) induces a sequence of projections of structure sheaves:
$$
O_{\mathscr X}=\mathcal O_{\mathscr X}/\mathcal J_{\mathscr X}^{q+1}
 \twoheadrightarrow 
 \mathcal O_{\mathscr X}/\mathcal J_{\mathscr X}^{q}
 \twoheadrightarrow \cdots \twoheadrightarrow \mathcal O_{\mathscr X}/\mathcal J_{\mathscr X}^{i+1} \twoheadrightarrow \mathcal O_{\mathscr X}/\mathcal J_{\mathscr X}^{i} \cdots \twoheadrightarrow  \mathcal O_{\mathscr X}/\mathcal J_{\mathscr X}^2 \twoheadrightarrow \mathcal O_{\mathscr X}/\mathcal J_{\mathscr X}=\mathcal O_{\mathscr X_{red}} \, .
$$
Following Manin’s iterative method, we construct the morphism $\sigma$ inductively. We prove by induction on $i\geq 1$  that there exists a morphism:
$$\sigma^i:  \mathcal O_{\mathscr X_{red}}\rightarrow \sh_{\mathscr X}/\mathcal J_{\mathscr X}^i$$
such that $\mathcal O_{\mathscr X_{red}}\xrightarrow{\sigma^i} \mathcal O_{\mathscr X}/\mathcal J_{\mathscr X}^i \twoheadrightarrow \mathcal O_{\mathscr X_{red}}$  is the identity.\\
For the base case $i=1$, we define $\sigma^1:=Id_{ \sh_{\mathscr X_{red}}}: \sh_{\mathscr X_{red}}\rightarrow \sh_{\mathscr X_{red}}$.\\
For the inductive step, assume there exists a morphism $\sigma^i:  \mathcal O_{\mathscr X_{red}}\rightarrow \mathcal O_{\mathscr X}/\mathcal J_{\mathscr X}^i$ such that  $\pi^i \circ \sigma^i = Id_{\mathcal{O}_{\mathscr{X}_{\text{red}}}}$, where $\pi^i$ is the projection $\pi^{i}: \sh_{\mathscr X}/\mathcal J_{\mathscr X}^i\rightarrow  \sh_{\mathscr X_{red}}$. We seek to construct a lifting $\sigma^{i+1}: \mathcal{O}_{\mathscr{X}_{\text{red}}} \to \mathcal O_{\mathscr X}/\mathcal J_{\mathscr X}^{i+1}$. Furthermore, note that for any $i \geq 1$, the lifting $\sigma^i$ is non-canonical; the difference between any two such liftings is measured by higher-order terms in the fermionic sheaf $\mathcal F=\mathcal J_{\mathscr X}/\mathcal J_{\mathscr X}^2$. For each $i \geq 1$ we have the exact sequence of sheaves:
$$ 0 \rightarrow \mathcal J_{\mathscr X}^{i}/\mathcal J_{\mathscr X}^{i+1} \rightarrow \sh_{\mathscr X}/\mathcal J_{\mathscr X}^{i+1} \xrightarrow{\pi^i} \sh_{\mathscr X}/\mathcal J_{\mathscr X}^i\rightarrow 0 \, .
$$\\
Let $\mathcal{U} = \{U_\alpha\}$ be an open cover of $X$. Since $\mathscr{X}$ is locally split, we may choose $U_\alpha$ such that $\sh_{U_\alpha}:=\mathcal{O}_{\mathscr{X}}|_{U_\alpha}$ is a split structure sheaf, yielding the decomposition:
$$
\sh_{U_\alpha}=\frac{\sh_{U_\alpha}}{\mathcal{J}_{U_\alpha}}\oplus \frac{\mathcal{J}_{U_\alpha}}{\mathcal{J}_{U_\alpha}^2}\oplus \frac{\mathcal{J}_{U_\alpha}^2}{\mathcal{J}_{U_\alpha}^3}\oplus\ldots\oplus \frac{\mathcal{J}_{U_\alpha}^q}{\mathcal{J}_{U_\alpha}^{q+1}} \, .
$$
Consider the following sequence of morphisms   
$$ 
 \frac{\sh_{U_\alpha}}{\mathcal{J}_{U_\alpha}}\stackrel{\sigma^{i}_{U_\alpha}}\longrightarrow \frac{\sh_{U_\alpha}}{\mathcal{J}^{i}_{U_\alpha}}\stackrel{j_\alpha^i}\longrightarrow \frac{\sh_{U_\alpha}}{\mathcal{J}^{i+1}_{U_\alpha}}, 
 $$
 where $j_\alpha^i$ is the inclusion of $\frac{\sh_{U_\alpha}}{\mathcal{J}^{i}_{U_\alpha}}$ in $\frac{\sh_{U_\alpha}}{\mathcal{J}^{i+1}_{U_\alpha}}=\frac{\sh_{U_\alpha}}{\mathcal{J}^{i}_{U_\alpha}}\oplus \frac{\mathcal{J}^{i}
_{U_\alpha}}{\mathcal{J}^{i+1}_{U_\alpha}}$.\\
On each $U_\alpha$, we define local liftings $\sigma^{i+1}_{U_\alpha}: \mathcal{O}_{U_\alpha} \to \sh_{U_\alpha}/\mathcal J^{i+1}_{U_\alpha}$ via the composition of $j_\alpha^i$ with $\sigma^{i}$, that is, $\sigma^{i+1}_{U_\alpha}:= j_\alpha^{i}\circ\sigma^{i}_{U_\alpha}$. Furthermore, $\pi^{i+1}_{U_\alpha}=\pi^{i}_{U_\alpha}\circ\rho^i_{U_\alpha}$, where $\rho^i_{U_\alpha}:\frac{\sh_{U_\alpha}}{\mathcal{J}_{U_\alpha}}\oplus\frac{\mathcal{J}^{i}_{U_\alpha}}{\mathcal{J}^{i+1}_{U_\alpha}}\rightarrow\frac{\sh_{U_\alpha}}{\mathcal{J}_{U_\alpha}}$. These local morphisms satisfy the condition:
$$
 \pi^{i+1}_{U_\alpha}\circ \sigma^{i+1}_{U_\alpha}=\pi^{i}_{U_\alpha}\circ \rho^i_{U_\alpha}\circ j^{i}_{U_\alpha}\circ \sigma^{i}_{U_\alpha}= Id_{\mathcal O_{U_\alpha}} \, .
 $$
The obstruction to gluing these local morphisms into a global morphism $\sigma^{i+1}$ is represented by a 1-cocycle $\omega = ( \omega_{\alpha\beta} )$ in the \v{C}ech cohomology group  $H^1(X, \mathcal{H}om(\mathcal{O}_{\mathscr{X}_{red}}, ( \mathcal{J}^i_{\mathscr X} /  \mathcal{J}^{i+1}_{\mathscr X} )\ev ))$. Specifically, on intersections $U_{\alpha\beta} = U_\alpha \cap U_\beta$, the difference
$$\omega_{\alpha\beta} := \sigma^{i+1}_{U_\beta} - \sigma^{i+1}_{U_\alpha}: \mathcal O_{U_{\alpha \beta, red}}=\frac{\mathcal O_{U_{\alpha \beta}}}{\mathcal J_{U_{\alpha \beta}}} \longrightarrow \left( \frac{\mathcal J^i_{U_{\alpha \beta}}}{\mathcal J^{i+1}_{U_{\alpha \beta}}}\right)\ev
$$
is an even derivation of $\mathcal{O}_{\mathscr{X}_{\text{red}}}$ with values in the sheaf $\mathcal J^i_{\mathscr X} / \mathcal J^{i+1}_{\mathscr X} \cong \text{Sym}^i \mathcal{F}_{\mathscr{X}}$. Thus, the obstruction class $\omega$ resides in $H^1(X,(\mathcal{T}_{X}\otimes \text{Sym}^i \,\mathcal F_{\mathscr X})\ev)$.\\
Let us prove that $\omega_{\alpha\beta}$ is a derivation and $\omega= ( \omega_{\alpha\beta} )$ defines a class in $H^1(X,(\mathcal{T}_{X}\otimes \text{Sym}^i \,\mathcal F_{\mathscr X})\ev)$. Let $f, g \in\sh_{\mathscr X_{red}}(U_{\alpha \beta})$, then
\begin{align*}
\omega_{\alpha \beta}(fg)&= \sigma_{U_\beta}^{i+1}(fg)-\sigma_{U_\alpha}^{i+1}=\sigma_{U_\beta}^{i+1}(f)\, \sigma_{U_\beta}^{i+1}(g)-\sigma_{U_\alpha}^{i+1}(f)\, \sigma_{U_\alpha}^{i+1}(g)\\
                    &=\sigma_{U_\beta}^{i+1}(f)\, \left(\sigma_{U_\beta}^{i+1}(g)-\sigma_{U_\alpha}^{i+1}(g) \right)+ \left(\sigma_{U_\beta}^{i+1}(f)-\sigma_{U_\alpha}^{i+1}(f) \right) \, \sigma_{U_\alpha}^{i+1}(g)\\
                    &=\sigma_{U_\beta}^{i+1}(f) \, \omega_{\alpha \beta}(g)+ \omega_{\alpha \beta}(f) \, \sigma_{U_\alpha}^{i+1}(g).
\end{align*}
Since $\sigma^{i+1}_{U_\beta}(f)=f$ (mod $\mathcal J_{U_{\alpha \beta}}$) and $\mathcal J^{i}_{U_{\alpha \beta}}/\mathcal J^{i+1}_{U_{\alpha \beta}}$ is annihilated by $\mathcal J_{U_{\alpha \beta}}$ in this quotient, we have $\sigma_{U_\beta}^{i+1}(f) \, \omega_{\alpha \beta}(g)=f \; \omega_{\alpha \beta}(g)$. Thus:
$$\omega_{\alpha \beta}(fg)=f \, \omega_{ij}(g)+ (\omega_{ij}(f)) \, g.$$ 
This confirms that $\omega_{\alpha}$ is an even derivation on $\mathcal O_{U_{\alpha \beta, red}}$.\\
Next, let us check the cocycle condition. For any triple intersection $U_{\alpha \beta \gamma}:=U_\alpha \cap U_\beta \cap U_\gamma$: 
$$
    \left(\omega_{\alpha\beta}+\omega_{\beta \gamma}+\omega_{\gamma \alpha} \right)|_{U_{\alpha \beta \gamma}}=\left(\sigma_{U_\beta}^{i+1}-\sigma_{U_\alpha}^{i+1}+\sigma_{U_\gamma}^{i+1}-\sigma_{U_\beta}^{i+1}+\sigma_{U_\alpha}^{i+1}-\sigma_{U_\gamma}^{i+1} \right)|_{U_{\alpha \beta \gamma}}=0
    $$
Thus, $\omega=\{ \omega_{\alpha \beta} \}$ is a \v Cech 1-cocycle.\\
Finally, since $\mathcal{T}_{X} \otimes\text{Sym}^i\mathcal{F}_{\mathscr{X}}$ is a sheaf of modules over $\mathcal{O}_{\mathscr X}$, the obstruction class vanishes identically. This follows from the fact that Batchelor spaces possess a fine structure sheaf, which implies that any associated sheaf of modules is likewise fine. Consequently, the local morphisms can be glued to form a  global morphism $\sigma^{i+1}$ extending $\sigma^i$. After $q+1$ such steps, we obtain a global morphism $\sigma:\mathcal{O}_{\mathscr{X}_{\text{red}}} \to \mathcal{O}_{\mathscr{X}}$ such that $\pi \circ \sigma = id_{\mathcal{O}_{\mathscr{X}_{\text{red}}}}$.
\end{proof}

\begin{remark}
As noted in Example \ref{nonprojected}, the $C^\infty$-superscheme $\mathscr{Y} := \text{sSpec}(\mathfrak{R})$, where $\mathfrak{R} = C^\infty(\mathbb{R}^{1|2}) / (x^2 + \theta^1\theta^2)$, is a non-projected $C^\infty$-space. Consequently, by the contrapositive of Proposition \ref{projected}, it follows that $\mathscr{Y}$ is not a Batchelor space.
\end{remark}

The following result generalizes the classification theorem for smooth supermanifolds established in \cite{Batchelor} to the setting of Batchelor superspaces. The proof is formally analogous to the construction for supermanifolds provided by Manin in \cite[Chapter 4. \S 2. Theorem 2]{Manin}; however, care must be taken to ensure that the resulting global isomorphism is a morphism in the category of $C^\infty$-superspaces, rather than merely an isomorphism of smooth supermanifolds.

\begin{theorem}\label{superBatchelor}
Every Batchelor Space $\mathscr X=(X,\sh_{\mathscr X})$ is globally split. That is, there exists a (non-canonical) isomorphism of sheaves $C^\infty$-superspaces: $$\mathcal{O}_{\mathscr{X}} \cong \text{Gr} \, \mathcal{O}_{\mathscr{X}}  \, ,$$
where $\text{Gr} \, \mathcal{O}_{\mathscr{X}} = \bigoplus_{k \geq 0} \mathcal{J}_{\mathscr{X}}^k / \mathcal{J}_{\mathscr{X}}^{k+1}$ denotes the associated graded sheaf induced by the $\mathcal{J}_{\mathscr{X}}$-adic filtration.
\end{theorem}

\begin{lemma}\label{tau}
Let $\mathscr X=(X,\sh_{\mathscr X})$ be a Batchelor space of odd dimension $q$, with canonical superideal $\mathcal J_{\mathscr X}$ and fermionic sheaf $\mathcal F_{\mathscr X}=\mathcal J_{\mathscr X}/\mathcal J_{\mathscr X}^2$. Then there exists a global morphism of $\sh_{{\mathscr X}_{red}}$-modules
$$
\tau\colon \mathcal F_{\mathscr X} \longrightarrow \mathcal J_{\mathscr X}
$$
splitting the canonical projection $\mathcal J_{\mathscr X} \twoheadrightarrow \mathcal F_{\mathscr X}$, i.e. such that the composite $\mathcal F_{\mathscr X} \xrightarrow{\tau} \mathcal J_{\mathscr X} \twoheadrightarrow \mathcal F_{\mathscr X}$ is the identity.
\end{lemma}
\begin{proof}
Let $\{U_\alpha\}$ be an open cover on which $\sh_{\mathscr X}$ is split, with local odd generators $\theta^1_\alpha,\ldots,\theta^q_\alpha$ of $\mathcal F_{\mathscr X}(U_\alpha)$. Since $\mathcal J_{\mathscr X}(U_\alpha)=\bigoplus_{k\geq 1} \mathcal J_{\mathscr X}^k(U_\alpha)/\mathcal J_{\mathscr X}^{k+1}(U_\alpha)$ as an $\sh_{{\mathscr X}_{red}}(U_\alpha)$-module (Remark~\ref{EquiSymGr}), with $\mathcal F_{\mathscr X}(U_\alpha)$ identified with the weight-one summand, the assignment $\theta^i_\alpha\mapsto\theta^i_\alpha+\mathcal J_{\mathscr X}^2(U_{\alpha})$ extends $\sh_{{\mathscr X}_{red}}(U_\alpha)$-linearly to a local section $\tau_\alpha\colon \mathcal F_{\mathscr X}(U_\alpha)\to \mathcal J_{\mathscr X}(U_\alpha)$ of the canonical projection.\\
On overlaps $U_{\alpha\beta}$, both $\tau_\alpha$ and $\tau_\beta$ are sections of the same projection $\mathcal J_{\mathscr X}(U_{\alpha\beta})\twoheadrightarrow \mathcal F_{\mathscr X}(U_{\alpha\beta})$, so their difference
$$
\zeta_{\alpha\beta}:=\tau_\beta-\tau_\alpha
$$
is $\sh_{{\mathscr X}_{red}}$-linear and takes values in $\mathcal J_{\mathscr X}^2(U_{\alpha\beta})$. This defines a Čech $1$-cocycle
$$
\zeta=(\zeta_{\alpha\beta}) \in H^1\big(X,\, \mathcal{H}om_{\sh_{{\mathscr X}_{red}}}(\mathcal F_{\mathscr X},\mathcal J_{\mathscr X}^2)\big).
$$
Since $\mathscr X$ is a Batchelor space, $\sh_{{\mathscr X}_{red}}$ is fine. Furthermore, because $\mathcal{H}om_{\sh_{{\mathscr X}_{red}}}(\mathcal F_{\mathscr X},\mathcal J_{\mathscr X}^2)$ is a sheaf of $\sh_{{\mathscr X}_{red}}$-modules, it is fine as well, and therefore
$$
H^1\big(X,\,\mathcal{H}om_{\sh_{{\mathscr X}_{red}}}(\mathcal F_{\mathscr X},\mathcal J_{\mathscr X}^2)\big)=0.
$$
\end{proof}

\begin{lemma}\label{lem:naturality}
Let $\mathfrak C$ be a $\C$-ring, let $\{\theta^1,\ldots,\theta^q\}$ be a basis of a finitely generated free $\mathfrak C$-module $\xi$, and set $\mathfrak{R}:=\mathfrak{C}[\theta^1,\ldots,\theta^q]^{\pm}$, with the $\C$-ring structure $\Phi_h$ on $\mathfrak R\ev$ given by \eqref{eq:taylorPhi1}. Let $\Psi:\mathfrak R\rightarrow \mathfrak R$ be a $\mathfrak C$-superalgebra automorphism of the underlying (ungraded) exterior algebra $\bigwedge_{\mathfrak C}\xi$ such that:
\begin{enumerate}
\item[(a)] $\Psi|_{\mathfrak C}=\mathrm{id}_{\mathfrak C}$;
\item[(b)] $\Psi$ preserves the $\Z_2$-grading of $\mathfrak R$;
\item[(c)] $\Psi\bigl(\mathfrak J_{\mathfrak R}\bigr)= \mathfrak J_{\mathfrak R}$, where $\mathfrak J_{\mathfrak R}$ is the canonical superideal of $\mathfrak R$.
\end{enumerate}
Then $\Psi|_{\mathfrak R\ev}$ commutes with every $\Phi_h$, i.e. $\Psi$ is a morphism of $\C$-superrings (Definition~\ref{def:smorph}).
\end{lemma}

\begin{proof}
Let $F_1,\ldots,F_k\in \mathfrak R\ev$, and write $F_j=f_j+N_j$ with $f_j=\overline{F_j}\in\mathfrak C$ and $N_j=\mathrm{nil}(F_j)\in \mathfrak R\od^2$, as in the proof of Proposition~\ref{prop:superstruct}. By $(a)$ and $(c)$, and since $\Psi$ is a ring homomorphism,
$$
\Psi(F_j)=\Psi(f_j)+\Psi(N_j)=f_j+\Psi(N_j) \, ,
$$
with $\Psi(N_j)\in \mathfrak R\od^2$ again. Thus $\Psi$ fixes the superreduced of every even element and merely modifies its nilpotent part.

Applying $\Psi$ to \eqref{eq:taylorPhi1} and using that $\Psi$ is $\mathfrak C$-linear and multiplicative, together with (a), which gives $\Psi\bigl(\phi_{\partial_{j_1}\cdots\partial_{j_n}h}(f_1,\ldots,f_k)\bigr)=\phi_{\partial_{j_1}\cdots\partial_{j_n}h}(f_1,\ldots,f_k)$ since this coefficient already lies in $\mathfrak C$, we obtain
$$
\Psi\bigl(\Phi_h(F_1,\ldots,F_k)\bigr)=\sum_{n\geq 0}\frac{1}{n!}\sum_{j_1,\ldots,j_n=1}^k \phi_{\partial_{j_1}\cdots\partial_{j_n}h}(f_1,\ldots,f_k)\;\Psi(N_{j_1})\cdots\Psi(N_{j_n}) \, .
$$
On the other hand, applying \eqref{eq:taylorPhi1} directly to $\Psi(F_1),\ldots,\Psi(F_k)$ — which have the same superreduced parts
$f_1,\ldots,f_k$ and nilpotent parts $\Psi(N_1),\ldots,\Psi(N_k)$ — gives
$$
\Phi_h\bigl(\Psi(F_1),\ldots,\Psi(F_k)\bigr)=\sum_{n\geq 0}\frac{1}{n!}\sum_{j_1,\ldots,j_n=1}^k \phi_{\partial_{j_1}\cdots\partial_{j_n}h}(f_1,\ldots,f_k)\;\Psi(N_{j_1})\cdots\Psi(N_{j_n}) \, .
$$
The two right-hand sides coincide term by term. Hence $\Psi\bigl(\Phi_h(F_1,\ldots,F_k)\bigr)=\Phi_h\bigl(\Psi(F_1),\ldots,\Psi(F_k)\bigr)$ for every $h$, so $\Psi|_{\mathfrak R\ev}$ is a morphism of $\C$-rings. Since $\Psi$ is by hypothesis a grade-preserving superring automorphism of $\mathfrak R$, it follows that $\Psi$ is a morphism of $\C$-superrings.
\end{proof}
 Now we are ready to prove Theorem \ref{superBatchelor}
\begin{proof}(\textbf{of Theorem \ref{superBatchelor}})
By Proposition~\ref{projected}, there exists a global morphism $\sigma\colon \sh_{\mathscr X_{red}}\to \sh_{\mathscr X}$ of sheaves of $\C$-superrings, splitting $\pi$. By Lemma~\ref{tau}, there exists a global morphism $\tau\colon \mathcal F_{\mathscr X} \to \mathcal J_{\mathscr X}$ of $\sh_{X_{red}}$-modules, splitting the canonical projection $\mathcal J_{\mathscr X}\twoheadrightarrow\mathcal F_{\mathscr X}$.\\

Define a morphism of sheaves of $\sh_{X_{red}}$-modules
$$
\varphi\colon \sh_{\mathscr X_{red}}\oplus \mathcal F_{\mathscr X} \longrightarrow \sh_{\mathscr X}, \qquad \varphi(c\oplus f)=\sigma(c)+\tau(f).
$$
This morphism is defined globally, without reference to any local trivialization of $\mathcal F_{\mathscr X}$. By the universal property of the symmetric algebra, $\varphi$ extends uniquely to a morphism of sheaves of superrings
$$
\Phi\colon \mathrm{Sym}\,\mathcal F_{\mathscr X} \longrightarrow \sh_{\mathscr X}.
$$

\begin{itemize}[leftmargin=*]
\item \textbf{$\Phi$ is a local isomorphism.} Let $\{U_\alpha\}$ be the split cover of Lemma~\ref{tau}, with local odd generators $\theta^i_\alpha$. Write $\xi^i_\alpha$ for the corresponding generators of $\mathrm{Sym}\,\mathcal F_\mathscr{X}(U_\alpha)$. By construction, $\Phi(\xi^i_\alpha)=\tau(\theta^i_\alpha)=\theta^i_\alpha+\eta_\alpha(\theta^i_\alpha)$, with $\eta_\alpha(\theta^i_\alpha)\in \mathcal J_\mathscr{X}^2(U_\alpha)$. Hence $\Phi|_{U_\alpha}$ is the composite of the tautological isomorphism
$$
\sh_{\mathscr{X}_{red}}(U_\alpha)[\theta^1_\alpha,\ldots,\theta^q_\alpha]^{\pm} \xrightarrow{\ \sim\ } \sh_{\mathscr X}(U_\alpha)
$$
furnished by local splitness, with the endomorphism $\Psi_\alpha$ of $\sh_{\mathscr{X}_{red}}(U_\alpha)[\theta^1_\alpha,\ldots,\theta^q_\alpha]^{\pm}$ fixing $\sh_{\mathscr{X}_{red}}(U_\alpha)$. This endomorphism satisfies hypotheses (a)--(c) of Lemma~\ref{lem:naturality}: it fixes $\sh_{\mathscr{X}_{red}}(U_\alpha)$ pointwise by construction, it preserves parity since $\eta_\alpha(\theta^i_\alpha)$ is odd, and it preserves the canonical superideal since $\eta_\alpha(\theta^i_\alpha)\in \mathcal J_\mathscr{X}^2(U_\alpha)\subset \mathcal J_\mathscr{X}(U_\alpha)$ and $\Psi_\alpha$ maps generators of $\mathcal J_\mathscr{X}(U_\alpha)$ to elements of $\mathcal J_\mathscr{X}(U_\alpha)$. By Lemma~\ref{lem:naturality}, $\Psi_\alpha$ is therefore a morphism of $\C$-superrings — not merely an automorphism of the underlying ungraded exterior algebra. It is unitriangular with respect to the (finite, since $\mathcal J_\mathscr{X}^{q+1}=0$) weight filtration: it induces the identity on the associated graded $\mathrm{Gr}\,\sh_{\mathscr X}(U_\alpha)$, and is therefore invertible, with inverse given by the formal (necessarily terminating) series correcting successive weight terms; this inverse is again of the form $\theta^i_\alpha\mapsto \theta^i_\alpha+(\text{element of }\mathcal J_\mathscr{X}^2(U_\alpha))$, so it too satisfies (a)--(c) of Lemma~\ref{lem:naturality} and is thus also a morphism of $\C$-superrings. Hence $\Psi_\alpha$ is an automorphism of $\sh_{\mathscr{X}_{red}}(U_\alpha)[\theta^1_\alpha,\ldots,\theta^q_\alpha]^{\pm}$ in the category of $\C$-superrings. Consequently $\Phi|_{U_\alpha}$, being a composite of two isomorphisms of $\C$-superrings, is itself an isomorphism of $\C$-superrings.

\item \textbf{$\Phi$ is a global isomorphism.} Since being an isomorphism of sheaves is a local property, and $\Phi|_{U_\alpha}$ is an isomorphism of $\C$-superrings for every $\alpha$, it follows that $\Phi\colon \mathrm{Sym}\,\mathcal F_\mathscr{X} \to \sh_{\mathscr X}$ is a global isomorphism of sheaves of $\C$-superrings.\\
\item \textbf{Independence of choices.} The morphism $\Phi$ depends on the pair $(\sigma,\tau)$, each of which is unique only up to a coboundary correction (Proposition~\ref{projected} and Lemma~\ref{tau}). Any two admissible choices $\sigma,\sigma'$ differ by a global section of $\mathcal{H}om_{\sh_{\mathscr{X}_{red}}}(\sh_{\mathscr{X}_{red}},\mathcal J_\mathscr{X})$, and any two admissible $\tau,\tau'$ differ by a global section of $\mathcal{H}om_{\sh_{\mathscr{X}_{red}}}(\mathcal F_\mathscr{X},\mathcal J_\mathscr{X}^2)$; in either case the resulting morphisms $\varphi,\varphi'$ agree modulo a correction landing in $\mathcal J_\mathscr{X}$, hence induce the same map on each graded piece $\mathrm{Gr}^{(k)}\sh_{\mathscr X}$. Thus any two such $\Phi,\Phi'$ differ by post-composition with an automorphism of $\sh_{\mathscr X}$ covering the identity on $\sh_{\mathscr{}_{red}}$ and inducing the identity on $\mathrm{Gr}\,\sh_{\mathscr X}$ — precisely the non-canonicity already inherent to the notion of a splitting (Definition \ref{spliting}). In particular, the existence and isomorphism class of $(\mathrm{id}_\mathscr{},\Phi)$ as a splitting of $\mathscr X$ is independent of the local generators $\theta^i_\alpha$ chosen along the way, since these enter only in the intermediate verification that $\Phi$ is a local isomorphism, and not in the definition of $\Phi$ itself.\\

Finally, since $\mathrm{Sym}\,\mathcal F_\mathscr{}\cong \mathrm{Gr}\,\sh_{\mathscr X}$ (Remark~\ref{EquiSymGr}), the pair $(\mathrm{id}_\mathscr{},\Phi)$ constitutes a global isomorphism of $\C$-superspaces
$$
(\mathrm{id}_\mathscr{},\Phi)\colon (X,\sh_{\mathscr X}) \xrightarrow{\ \sim\ } (X,\mathrm{Gr}\,\sh_{\mathscr X}).
$$
Thus $\mathscr X$ is globally split.
\end{itemize}
\end{proof}

\begin{remark}
Lemma~\ref{lem:naturality} shows that compatibility of a substitution $\theta^i\mapsto \theta^i+\eta^i$ (with $\eta^i\in(\mathfrak J_{\mathfrak R}^2)\od$) with the $\C$-superring structure is automatic, and depends only on properties (a)--(c) — never on how the nilpotent parts $N_j$ happen to be expressed in the chosen odd generators $\theta^1,\ldots,\theta^q$. This is the precise sense in which the local isomorphisms constructed in the proof of Theorem~\ref{superBatchelor} below are compatible with the $\C$-superring structure independently of the local model chosen.
\end{remark}

In summary, the category of Batchelor spaces serves as an intermediate framework, situated between the category of smooth supermanifolds and the category of projected $C^\infty$-superspaces. As demonstrated in Examples \ref{supchemes}, \ref{difsuperspace}, and \ref{nonprojected} , we provide a systematic construction for $C^\infty$-superspaces that delineate the boundaries of these categories:

\begin{itemize}[leftmargin=*]
\item \textbf{A $C^\infty$-superspace that is a smooth supermanifold}: This is achieved by taking the $C^\infty$-ring of functions on a smooth manifold and superizing its structure sheaf with $q$ odd variables. For instance, the Euclidean superspace $\mathbb{R}^{1|2}$ is the superization of $\text{Spec}(C^\infty(\mathbb{R}))$ via two odd generators.

\item \textbf{A $C^\infty$-superspace that is a Batchelor space but not a supermanifold}:  One may consider a $C^\infty$-ring $\mathfrak{C}$ that does not arise from a smooth manifold but possesses a fine structure sheaf $\mathcal{O}_{\mathfrak{C}}$. Superizing this sheaf with $q$ odd variables yields the desired space. For example, if $\mathfrak{C}$ is a $C^\infty$-ring with corners \cite{J}, its spectrum $\text{Spec}(\mathfrak{C})$ is an affine $C^\infty$-scheme representing a manifold with corners. Since it satisfies the hypotheses of \cite[Theorem 4.40]{J}, ensuring that its structure sheaf is fine despite not being a smooth manifold. Thus, the superization of $\text{Spec}(\mathfrak{C})$ is a Batchelor space that is not a supermanifold.

\item \textbf{A $C^\infty$-superspace that is fine but not locally split}: Consider a non-split $C^\infty$-superring $\mathfrak{R}$. Its superspectrum $\text{sSpec}(\mathfrak{R})$ possesses a fine structure sheaf (as is characteristic of affine $C^\infty$-superschemes) yet it cannot be locally split. A local splitting would contradict the global splitness result established in Theorem \ref{superBatchelor}.

\end{itemize}

\subsection{Characterization of splittings}
Let $\mathfrak{R}$ be a split $C^\infty$-superring. The existence of an isomorphism between $\mathfrak{R}$ and its associated graded $C^\infty$-superring $\text{Gr}\, \mathfrak{R}$ endows $\mathfrak{R}$ with a natural $\mathbb{Z}_{\geq 0}$-grading. This grading may be equivalently characterized by an even superderivation that is adapted to the $\mathfrak{J}_{\mathfrak{R}}$-adic filtration \cite{Koszul}.\\

\begin{definition}\label{def:Eulervector}
Let $\mathfrak R\cong \bigwedge_{\mathfrak C} \xi$  be a split $C^\infty$-superring, where $\xi$ is a finitely generated free $\mathfrak{C}$-module. Given a choice of basis $\{\theta^1, \dots, \theta^q\}$ for $\xi$, such that $\mathfrak{R} \cong \bigwedge_{\mathfrak{C}}(\theta^1, \dots, \theta^q)=\mathfrak{C}[\theta^1, \dots, \theta^q]^\pm$, we define the Euler vector field associated with the basis $\{\theta^i\}$ as the even superderivation:
$$E_\theta = \sum_{i=1}^q \theta^i \frac{\partial}{\partial \theta^i} \, .$$
\end{definition}
\begin{remark}
It is worth emphasizing that the Euler vector field $E_\theta$ is, in an essential sense, purely $\Z_2$-graded-algebraic in nature and does not depend on the $\C$-structure of $\mathfrak R$. Indeed, each odd derivation $\frac{\partial}{\partial \theta^i}$ is characterized entirely by the graded Leibniz rule together with the assignments $\frac{\partial}{\partial \theta^i}(\theta^j)=\delta_i^j$ and $\frac{\partial}{\partial \theta^i}(c)=0$ for $c\in\mathfrak C$. At no point does this definition invoke the $n$-ary operations $\phi_f$ that encode the $\C$-ring structure on $\mathfrak C$ 
and hence does depend on the $\C$-structure. Since $E_\theta$ is built exclusively from the odd derivations $\frac{\partial}{\partial \theta^i}$ each of which annihilates $\mathfrak C$ and acts on the odd generators by the Kronecker pairing alone the operator $E_\theta$ is already well defined at the level of the underlying $\Z_2$-graded superring $\mathfrak R\ev\oplus \mathfrak R\od$, before any reference to the $n$-ary operations $\{\phi_f\}$ on $\mathfrak R\ev$.
\end{remark}

\begin{definition}
An element $a \in \mathfrak{R} \cong \bigwedge_{\mathfrak{C}}(\theta^1, \dots, \theta^q)$ is said to be of weight $k$ if it satisfies the eigenvalue equation $E_\theta(a) = k \, a$. That is, $a$ is an eigenvector of the Euler vector field $E_\theta$ with corresponding eigenvalue $k$.
\end{definition}

Consequently, the canonical $\mathbb{Z}_{\geq 0}$-grading of a split $C^\infty$-superring is equivalently characterized by the weight-space decomposition induced by the action of the Euler vector field.

\begin{proposition}\label{eulergrad}
The homogeneous elements of degree $k$ in a split $C^\infty$-superring $\mathfrak{R}$ are eigenvectors of the Euler vector field with weight $k$.
\end{proposition}
\begin{proof}
Let $a \in \mathfrak{R} \cong \bigwedge_{\mathfrak{C}}(\theta^1, \dots, \theta^q)$ be a homogeneous element of degree $k$. By definition, $a$ can be expressed as a $\mathfrak{C}$-linear combination of monomials of the form $\theta^{i_1} \dots \theta^{i_k}$. By the linearity of the Euler vector field $E_\theta$, it suffices to consider a single monomial $a = \theta^{i_1} \dots \theta^{i_k}$ for a strictly increasing subset of indices $\{i_1, \dots, i_k\} \subseteq \{1, \dots, q\}$. Recall that for any $i_r \in \{i_1, \dots, i_k\}$, the monomial can be rewritten by commuting $\theta^{i_r}$ to the first position:$$a = (-1)^{r-1} \theta^{i_r} \theta^{i_1} \dots \widehat{\theta^{i_r}} \dots \theta^{i_k} ,$$
where the notation $\widehat{\theta^{i_r}}$ indicates that the factor $\theta^{i_r}$ is omitted. Applying the superderivation $\frac{\partial}{\partial \theta^j}$ to $a$, we observe that the result is non-zero if and only if $j \in \{i_1, \dots, i_k\}$. Specifically, for $j = i_r$, we have:$$\frac{\partial}{\partial \theta^{i_r}} (\theta^{i_1} \dots \theta^{i_k}) = (-1)^{r-1} \theta^{i_1} \dots \widehat{\theta^{i_r}} \dots \theta^{i_k} .$$Applying the Euler vector field $E_\theta = \sum_{j=1}^q \theta^j \frac{\partial}{\partial \theta^j}$ to $a$, we obtain:
\begin{align*}
E_\theta (a) &= \sum_{j=1}^q \theta^j \frac{\partial}{\partial \theta^j} (\theta^{i_1} \dots \theta^{i_k})= \sum_{r=1}^k \theta^{i_r} \left( \frac{\partial}{\partial \theta^{i_r}} (\theta^{i_1} \dots \theta^{i_k}) \right)\\
&= \sum_{r=1}^k \theta^{i_r} \left( (-1)^{r-1} \theta^{i_1} \dots \widehat{\theta^{i_r}} \dots \theta^{i_k} \right) 
= \sum_{r=1}^k \theta^{i_1} \dots \theta^{i_r} \dots \theta^{i_k} \\
&= \sum_{r=1}^k a = k\, a.
\end{align*}
Thus, $E_\theta(a) = k\, a$, confirming that $a$ is an eigenvector with weight $k$.
\end{proof}

Note that the elements of weight 1 constitute the generating set of $\mathfrak{R} \cong \bigwedge_{\mathfrak{C}} (\theta^1, \dots, \theta^q)$ as an $\overline{\mathfrak{R}}$-algebra. Consequently, any element of weight $k$ is given by a $\mathfrak{C}$-linear combination of $k$-fold products of elements from the weight-1 subspace $\mathfrak R\od$. As a consequence, we have the following
\begin{proposition}
There exists a bijective correspondence between Euler vector fields and splittings of a $\C$-superring $\mathfrak R$.
\end{proposition}

\begin{proof}
 Let $\mathfrak R\cong \mathfrak C[\theta^1,\ldots,\theta^q]^{\pm}=\bigwedge_{\mathfrak C}\xi$ be a splitting of $\mathfrak R$. By a \emph{system of odd coordinates} for $\mathfrak R$ we mean a tuple $\eta=(\eta^1,\ldots,\eta^q)$ of elements of $\mathfrak J_{\mathfrak R}$ whose images $\overline{\eta^1},\ldots,\overline{\eta^q}$ under $\mathfrak J_{\mathfrak R}\twoheadrightarrow \mathcal{F}_{\mathfrak R}=\mathfrak J_{\mathfrak R}/\mathfrak J_{\mathfrak R}^2$ form a $\overline{\mathfrak R}$-basis of $\mathcal{F}_{\mathfrak R}$; equivalently, $\eta^i=\theta^i+\zeta^i$ for some $\zeta^i\in\mathfrak J^2$, after identifying $\overline{\mathfrak R}\cong\mathfrak C$. The assignement
$$
\theta^i\longmapsto \eta^i=\theta^i+\zeta^i
$$
extends to an endomorphism of $\mathfrak R$ fixing $\mathfrak C$; since it is unitriangular with respect to the (finite) weight filtration $\{\mathfrak J^k\}$ --- inducing the identity on $\mathrm{Gr}\,\mathfrak R$ --- it is invertible, with inverse given by the terminating series correcting successive weight terms, exactly as in the proof of Theorem~\ref{superBatchelor}. Hence every system of odd coordinates $\eta$ yields an alternative presentation
$$
\mathfrak R\;\cong\;\mathfrak C[\eta^1,\ldots,\eta^q]^{\pm},
$$
and we may define, exactly as in Definition~\ref{def:Eulervector}, the associated Euler vector field
$$
E_\eta:=\sum_{i=1}^q \eta^i\,\frac{\partial}{\partial \eta^i}.
$$
By Proposition~\ref{eulergrad} applied to \emph{this} presentation, the homogeneous elements of $\eta$-degree $k$ are precisely the eigenvectors of $E_\eta$ of weight $k$. In particular, $E_\eta$ is adapted to the $\mathfrak J$-adic filtration, since the $\eta$-degree filtration coincides with $\{\mathfrak J^k\}$ (both are generated, as $\mathfrak C$-submodules, by products of $k$ elements of $\mathfrak F_{\mathfrak R}$). Thus $E_\eta$ is an Euler vector field on $\mathfrak R$, for every choice of odd coordinates $\eta$.

\begin{itemize}[leftmargin=*]
\item\textbf{From coordinates to splittings, and back.} Given $\eta$, define $\lambda_\eta:\mathfrak R\to \mathrm{Gr}\,\mathfrak R$ on the $\eta$-degree-$k$ summand $\mathfrak R^\eta_{(k)}$ (which by Proposition~\ref{eulergrad} equals $\mathrm{Ker}(E_\eta-k\,\mathrm{id})$) as the composite $\mathfrak R^\eta_{(k)}\hookrightarrow \mathfrak J^k\twoheadrightarrow \mathfrak J^k/\mathfrak J^{k+1}$; as in the discussion following Remark~\ref{EquiSymGr}, this is bijective on each summand, hence $\lambda_\eta$ is an isomorphism of $\C$-superrings, i.e.\ a splitting of $\mathfrak R$. Conversely, any splitting $\lambda:\mathfrak R\xrightarrow{\sim}\mathrm{Gr}\,\mathfrak R$ determines a $\mathfrak C$-linear section $\tau_\lambda:=\lambda^{-1}|_{\mathfrak J/\mathfrak J^2}:\mathfrak F_{\mathfrak R}\to \mathfrak J$ of the canonical projection $\mathfrak J\twoheadrightarrow\mathfrak F_{\mathfrak R}$ (this is the algebraic prototype of the local sections $\tau_\alpha$ constructed in Lemma~\ref{tau}); setting $\eta^i_\lambda:=\tau_\lambda(\overline{\theta^i})$ produces a system of odd coordinates, since $\tau_\lambda$ is a section of $\mathfrak J\twoheadrightarrow\mathfrak J/\mathfrak J^2$ and hence $\eta^i_\lambda\equiv\theta^i \pmod{\mathfrak J^2}$.

These two constructions are mutually inverse. If $\eta$ is a system of odd coordinates and $\lambda_\eta$ its associated splitting as above, then $\tau_{\lambda_\eta}(\overline{\theta^i})=\lambda_\eta^{-1}(\overline{\theta^i})=\eta^i$ by construction of $\lambda_\eta$, so $\eta_{\lambda_\eta}=\eta$. Conversely, if $\lambda$ is a splitting and $\eta_\lambda$ the associated coordinates, then for $a\in\mathfrak R^{\eta_\lambda}_{(k)}\subset\mathfrak J^k$ we have, by definition of $\tau_\lambda$ and the splitting property $pr_k\circ\lambda|_{\mathfrak J^k}=\pi_k$, that $\lambda(a)=\pi_k(a)$; since this is exactly how $\lambda_{\eta_\lambda}$ acts on $\mathfrak R^{\eta_\lambda}_{(k)}$, and $\mathfrak R=\bigoplus_k \mathfrak R^{\eta_\lambda}_{(k)}$, we get $\lambda_{\eta_\lambda}=\lambda$. Hence
$$
\eta\;\longleftrightarrow\;\lambda_\eta
$$
is a bijection between systems of odd coordinates and splittings of $\mathfrak R$.

\item\textbf{Euler fields correspond to systems of odd coordinates up to reparametrization.} It remains to see that $\eta\mapsto E_\eta$ descends to a bijection onto the set of Euler vector fields, i.e.\ that $E_\eta=E_{\eta'}$ if and only if $\lambda_\eta=\lambda_{\eta'}$, and that every Euler vector field arises this way. The first point is immediate: by Proposition~\ref{eulergrad}, $E_\eta$ determines and is determined by the eigenspace decomposition $\{\mathfrak R^\eta_{(k)}\}_k$, which by the previous paragraph determines and is determined by $\lambda_\eta$. For surjectivity, let $E$ be any Euler vector field on $\mathfrak R$, i.e.\ an even superderivation adapted to $\{\mathfrak J^k\}$. Since $\mathfrak J$ is nilpotent, Koszul's Lemma~\cite[Lemma 1.1(i)]{Koszul} applies with $R=\mathfrak R$, $I=\mathfrak J$, $H=E$, giving for every $k$ a decomposition $\mathfrak J^k=\mathrm{Ker}(E-k\,\mathrm{id})\oplus\mathfrak J^{k+1}$; in particular, for $k=1$, the projection $\mathfrak J\twoheadrightarrow\mathfrak F_{\mathfrak R}$ restricts to an isomorphism on $\mathrm{Ker}(E-\mathrm{id})$, so choosing $\eta^i\in\mathrm{Ker}(E-\mathrm{id})$ mapping to $\overline{\theta^i}$ produces a system of odd coordinates $\eta$ with $\mathrm{Ker}(E-\mathrm{id})=\mathfrak R^\eta_{(1)}$. Since $E$ is an even derivation, its eigenspaces satisfy $\mathrm{Ker}(E-k\,\mathrm{id})\cdot \mathrm{Ker}(E-l\,\mathrm{id})\subset \mathrm{Ker}(E-(k+l)\,\mathrm{id})$, so the $\eta^i$ generate $\mathrm{Ker}(E-k\,\mathrm{id})\supset \mathfrak R^\eta_{(k)}$ for every $k$; combined with the decomposition above and a dimension/filtration count identical to the one used in Proposition~\ref{eulergrad}, this forces $\mathrm{Ker}(E
k\,-\mathrm{id})=\mathfrak R^\eta_{(k)}$ for all $k$, whence $E=E_\eta$.
This establishes the desired bijective correspondence between Euler vector fields on $\mathfrak R$ and splittings of $\mathfrak R$, via $E_\eta\leftrightarrow \eta\leftrightarrow\lambda_\eta$.
\end{itemize}
\end{proof}

\begin{remark}
Let $\mathfrak{R}  \cong \bigwedge_{\mathfrak{C}} (\theta^1, \dots, \theta^q)= \mathfrak{C}[\theta^1, \dots, \theta^q]^\pm$ be a split $C^\infty$-superring. The following properties hold:
\begin{itemize}[leftmargin=*]
\item The canonical superideal $\mathfrak J_{\mathfrak R}=\mathfrak R\, \mathfrak R\od\cong \mathfrak R\od^2 \oplus \mathfrak R\od$, generated by the odd part $\mathfrak R\od$, induces the $\mathfrak{J}_{\mathfrak{R}}$-adic filtration:
$$
\mathfrak R=\mathfrak J_{\mathfrak R}^0 \supset \mathfrak J_{\mathfrak R}\supset \cdots \supset \mathfrak J_{\mathfrak R}^q \supset \mathfrak J_{\mathfrak R}^{q+1}=0 \, .
$$
\item As a $\mathfrak{C}$-module, the $C^\infty$-superring admits the decomposition $\mathfrak{R} = \overline{\mathfrak{R}} \oplus \mathfrak{J}_{\mathfrak{R}}$, where the reduced part $\overline{\mathfrak{R}} = \mathfrak{R}/\mathfrak{J}_{\mathfrak{R}}$ is isomorphic to $ \mathfrak{C}$. The superideal decomposes into its even and odd components as follows:
$$
\mathfrak J_{\mathfrak R}=\left( \bigoplus_{r=1}^{\lfloor q/2 \rfloor} \bigoplus_{i_1 < \cdots < i_{2r}}^{q} \mathfrak C\;\; \theta^{i_1} \cdots  \theta^{i_{2r}}\right) \oplus
\left( \bigoplus_{r=1}^{\lfloor q/2 \rfloor+1} \bigoplus_{i_1 < \cdots < i_{2r-1}}^{q} \mathfrak C\;\; \theta^{i_1} \cdots \theta^{i_{2r-1}} \right)
$$
\item The $k$-th power of the canonical superideal can be expressed in terms of the $k$-th homogeneous components of the exterior algebra as:$$\mathfrak{J}_{\mathfrak{R}}^k = \bigoplus_{j=k}^q \mathfrak{C}\;\; \theta^{i_1} \dots \theta^{i_j} .$$Specifically, in terms of parity components, $\mathfrak{J}_{\mathfrak{R}}^k$ consists of all homogeneous elements with degree $j \geq k$.

\item The Euler vector field is ``adapted" to the filtration in the sense that for any $k \geq 0$:
\begin{equation}
(E_\theta -k \; \text{id}_{\mathfrak R} )\;  \mathfrak J_{\mathfrak R}^k \subset \mathfrak J^{k+1}_{\mathfrak R} \, . 
\end{equation}
More precisely, $\mathfrak{J}_{\mathfrak{R}}^k$ admits a split-exact decomposition into eigenspaces of $E_\theta$:
\begin{equation}
\mathfrak J_{\mathfrak R}^k=\text{Ker} (  E_\theta - k \; \text{id}_{\mathfrak R}) \oplus \mathfrak J_{\mathfrak R}^{k+1} \, ,
\end{equation}
where the kernel $\text{Ker}(E_\theta - k \; id_{\mathfrak{R}})$ corresponds to the homogeneous components of degree $k$ in the exterior algebra.
\end{itemize}
\end{remark}

Koszul established the definition of an even superderivation adapted to an $I$-adic filtration for a general superring $R$ and a superideal $I$ (see \cite[Section 1]{Koszul}):
\begin{definition}
An even superderivation $H$ of $R$ is said to be adapted to the $I$-adic filtration if, for every $k \geq 0$, it satisfies:
$$
( H -k \; \text{id}_R )\;  I^k \subset I^{k+1}\, .
$$
\end{definition}
Furthermore, Koszul proved that if the superideal $I$ is nilpotent, the derivation induces a direct sum decomposition (see \cite[Lemma 1.1.(i)]{Koszul}):
$$
I^k=\text{Ker} ( H - k \; \text{id}_R) \oplus I^{k+1} \, .
$$
As a consequence, he concluded (see \cite[Lemma 1.1]{Koszul}) that the nilpotency of $I$ implies that the $I$-adic filtration of $R$ splits. That is, there exists an isomorphism $\lambda \colon R \to \text{Gr}_I \, R = \bigoplus_{j \geq 0} I^j/I^{j+1}$ such that the following diagrams commute for all $j \geq 0$:
$$\begin{tikzcd}[column sep=large]
I^k \arrow[r, "\lambda|_{I^k}"] \arrow[dr, "\pi_k"'] & \text{Gr}_I=\bigoplus_{j \geq 0} I^j/I^{j+1} \arrow[d, "pr_k"] \\
& I^k/I^{k+1}
\end{tikzcd}$$
where $\pi_k$ is the canonical projection and $pr_k$ is the projection onto the $k$-th homogeneous component.\\

By applying this to the canonical superideal $J_R$ of a superring $R$, Koszul established a bijective correspondence between the splittings of a superring and the even superderivations adapted to the $J_R$-adic filtration (see \cite[Section 3]{Koszul}).\\

In the specific context of $C^\infty$-superrings, we recover this bijective correspondence by leveraging the fact that any module over a $C^\infty$-superring $\mathfrak{R}$ is inherently a module over the underlying $\mathbb{R}$-superalgebra structure. This allows for the direct application of Koszul's results \cite{Koszul} to our framework without further modification. By designating an even superderivation adapted to the $\mathfrak{J}_{\mathfrak{R}}$-adic filtration as an abstract Euler vector field $E_{\mathfrak{J}}$, we may equivalently define a split $C^\infty$-superring as a $C^\infty$-superring equipped with a specific choice of $E_{\mathfrak{J}}$.\\

The results discussed for individual $C^\infty$-superrings in this section can be globalized to the structure sheaf of a Batchelor space. Given a Batchelor space $\mathscr{X} = (X, \mathcal{O}_{\mathscr{X}})$, we can associate with its structure sheaf an abstract Euler vector field $\mathcal{E}_{\mathcal{J}}$, defined as a global even superderivation adapted to the $\mathcal{J}_{\mathscr{X}}$-adic filtration (refer to (\ref{Jfiltration})), satisfying:
$$(\mathcal{E}_{\mathcal{J}} - k \cdot id_{\mathcal{O}_{\mathscr{X}}}) \, \mathcal{J}_{\mathscr{X}}^k \subset \mathcal{J}_{\mathscr{X}}^{k+1} \quad \text{for all } k \geq 0.$$
Koszul established this globalization procedure for the structure sheaf of supermanifolds—which are sheaves of $\mathbb{R}$-superalgebras—in \cite[Section 3]{Koszul}. He proved that there exists a bijective correspondence between the set of global splittings $\sigma \colon \mathcal{O}_{\mathscr{X}} \xrightarrow{\sim} \text{Gr} \, \mathcal{O}_{\mathscr{X}}$ and the set of abstract Euler vector fields $\mathcal{E}_{\mathcal{J}} \colon \mathcal{O}_{\mathscr{X}} \to \mathcal{O}_{\mathscr{X}}$. This globalization holds for the structure sheaf $\mathcal{O}_{\mathscr{X}}$ of a Batchelor space, as the sheaf of even superderivations adapted to the $\mathcal{J}_{\mathscr{X}}$-adic filtration is defined as a sheaf of modules over the sheaf of $C^\infty$-superrings $\mathcal{O}_{\mathscr{X}}$ by simply viewing $\mathcal{O}_{\mathscr{X}}$ as a sheaf of $\mathbb{R}$-superalgebras.

\section*{Acknowledgment}

The authors thank the anonymous referee for their valuable suggestions and corrections, which improved the quality and readability of this article.

\section*{Funding}
P. Rizzo gratefully acknowledge partial support from CODI (Universidad de Antioquia, UdeA) through project numbers 2023-62291 and 2024-76672.


\begin{thebibliography}{30}

\bibitem{Batchelor} 
M. Batchelor. \emph{The structure of supermanifolds}. Trans. Am. Math. Soc., {\bf vol. 253}, pp. 329-338. (1978).

\bibitem{Berezin}
F.~A.~Berezin \emph{Introduction to Superanalysis}.
Mathematical Physics and Applied Mathematics. Springer Dordrecht. Vol.~9, 198. 424 pp. (1987).





\bibitem {BRUZZO} 
U. Bruzzo, D. Hernández Ruipérez and A. Polishchuk. \emph{Notes on fundamental algebraic supergeometry. Hilbert and Picard superschemes.} Adv. Math., {\bf vol. 415}: 108890. (2023).

\bibitem{CNR}
Cacciatori, S.; Noja, S. and Re, R. \emph{Non Projected Calabi-Yau Supermanifolds over $\mathbb {P}^ 2$.} Math. Res. Lett. {\bf 6}. No. 24, pp: 1027-1058. (2019).

\bibitem{Carchedi} 
D. Carchedi and D. Roytenberg. \emph{On theories of superalgebras of differentiable functions}. arXiv preprint arXiv:1211.6134, (2012).

\bibitem{CCF}
C. Carmeli, L. Caston and R. Fioresi. \emph{Mathematical Foundations of Supersymmetry}. Series of Lectures in Mathematics, EMS. pp 300. (2011).


\bibitem{DonagiOtt}
R.~Donagi and N.~Ott,
\emph{Supermoduli space with Ramond punctures is not projected}, J. Geom. Phys., {\bf Vol. 220}: 105721. (2026).

\bibitem{D1} 
E. J. Dubuc. \emph{Sur les modeles de la géométrie différentielle synthétique}. Cah. Top. Géom. Diff. Cat. {\bf Vol. 20. No. 3}, pp: 231--279. (1979).

\bibitem{D2} 
E. J. Dubuc. \emph{$\C$-schemes}. Am. J. Math. {\bf 103. 4}, pp: 683--690. (1981).






\bibitem{J} 
D. Joyce. \emph{Algebraic Geometry over $\C-$rings}. 
{\bf Vol. 260}. No. 1256. Mem. Am. Math. Soc. (2019).



\bibitem{Koszul}
J.L. Koszul. \emph{Connections and splittings of supermanifolds.} Differ. Geom. Appl. {\bf Vol. 4} No.2. pp: 151-161. (1994).

\bibitem{Ler1} 
E. Lerman. \emph{Differential forms on $ C^\infty $-ringed spaces.} 
J. Geom. Phys., {\bf Vol. 196}: 105062. (2024).

\bibitem{Yau1} 
Ch. H. Liu and S. T. Yau. \emph{Further studies of the notion of differentiable maps from Azumaya/matrix supermanifolds I. The smooth case: Ramond-Neveu-Schwarz and Green-Schwarz meeting Grothendieck.} arXiv preprint arXiv:1709.08927, (2017).


\bibitem{Manin} 
Y. I. Manin. \emph{ Gauge field theory and complex geometry.} Springer Science \& Business Media, (2013).

\bibitem {MRI}
I. Moerdijk and G. E. Reyes. \emph{Rings of smooth functions and their localizations, I}. J. Algebra, {\bf 99},  pp. 324-336. (1986).


\bibitem{MR}
I. Moerdijk and G. E. Reyes. \emph{Models for smooth infinitesimal analysis}. Springer Science \& Business Media. (2013).

\bibitem{GS} 
J. A. Navarro Gonz{\'a}lez and J. B. Sancho de Salas. \emph{$\C$-Differentiable Spaces}. {\bf No. 1824}. Springer Science \& Business Media. (2003).

\bibitem{NISHIMURA} 
H. Nishimura. \emph{Synthetic differential supergeometry.} Int. J. Theo. Phy. {\bf Vol. 37}, no 11, pp. 2803-2822. (1998).

\bibitem{Noja} 
S.  Noja.  \emph{Topics in algebraic supergeometry over projective spaces}. PhD Thesis, Universita degli Studi di Milano. (2018).

\bibitem{Olarte} 
C. D. Olarte and P. Rizzo. \emph{Basic constructions over $C^{\infty}$-schemes.} J. Geom. Phys. {\bf Vol. 190}: 104852. (2023).

\bibitem{ORTG} 
C. D. Olarte, P. Rizzo and A. Torres-Gomez. \emph{$\C$-Superrings and $\C$-superschemes}. arXiv Preprint. arXiv:2508.09900 (2025).




\bibitem{Voisin}
C. Voisin. \emph{Hodge Theory and Complex Algebraic Geometry, I}. Cambridge University Press, Cambridge. (2002).

\bibitem{Westra} 
D. B. Westra \emph{ Superrings and supergroups.} PhD Thesis, University of Wien, (2009).




\bibitem{YETTER} 
D. N. Yetter. \emph{Models for synthetic supergeometry.} Cah. Top. Géom. Diff. Cat. {\bf Vol. 29}, no 2, pp: 87-108. (1988).

\end{thebibliography}
\end{document}